\documentclass[onesided]{article}
\RequirePackage{amsmath,amsfonts,amssymb}
\RequirePackage{mathtools}
\RequirePackage{graphicx}
\RequirePackage{pstricks}
\RequirePackage{epic,eepic}
\RequirePackage{color}
\RequirePackage{multicol}
\RequirePackage{longtable}
\RequirePackage{ifthen}
\RequirePackage{calc}
\RequirePackage[all,cmtip]{xy}

\let\dim\relax
\DeclareMathOperator{\rank}{rank}
\DeclareMathOperator{\dim}{dim}

\let\im\relax

\DeclareMathOperator{\im}{im}

\DeclareMathOperator{\Hom}{Hom}
\DeclareMathOperator{\id}{id}
\DeclareMathOperator{\Aut}{Aut}

\newcommand{\gen}[1]{\left\langle#1\right\rangle} 
\newcommand{\then}{\Rightarrow}

\renewcommand{\iff}{\Leftrightarrow}

\newcommand{\ds}{\displaystyle}

\newcommand{\zz}{\mathbb{Z}}

\newcommand{\rr}{\mathbb{R}}

\usepackage{fullpage} 

\newtheorem{lemma}{Lemma}
\newtheorem{theorem}{Theorem}
\newtheorem{proposition}{Proposition}

\newenvironment{proof}{\medskip\noindent\textit{Proof: \/}}
               {\begin{flushright}\rule{2mm}{2mm}\end{flushright}\par\medskip}
\newenvironment{remark}{\medskip\noindent\textbf{Remark: }}
               {\par\medskip}

\begin{document}

\title{Diagonal approximation and the cohomology ring\\ of torus fiber bundles}
\date{}
\author{S\'ergio Tadao Martins}
\maketitle

\abstract{For a torus bundle $(S^1\times S^1)\to E \to S^1$, we
  construct a finite free resolution of $\zz$ over $\zz[\pi_1(E)]$ and
  compute the cohomology groups $H^*(\pi_1(E), \zz)$ and
  $H^*(\pi_1(E),\zz_p)$ for a prime $p$. We also construct a partial
  diagonal approximation for the resolution, which allows us to
  compute the cup product in $H^*(\pi_1(E), \zz)$ and
  $H^*(\pi_1(E),\zz_p)$.}

\section{Introduction}

The calculation of the cohomology of a given group, including its ring
structure, is relevant for many applications and is considered an
interesting problem in its own right. For finite groups, for example,
see~\cite{Adem2004}. Another example are the manifolds that are
$K(\pi,1)$ spaces, in special many closed $3$-manifolds. These
manifolds are all $K(\pi,1)$ spaces with the exception of the
spherical ones and four $3$-manifolds that are covered by
$S^2\times\rr$ (see~\cite{Scott}), and so the cohomology rings of
these manifolds coincide with the cohomology rings of their
fundamental groups.

In addition, for a group $\pi$, the efficient calculation of the
multiplicative structure given by the cup product in $H^*(\pi,M)$ for
various coefficients $M$ is most easily accomplished by obtaining a
diagonal approximation for a projective resolution of $\zz$ over $\zz
\pi$. This was done, for example, by Tomoda and Zvengrowski
in~\cite{TomodaPeter2008}, where they computed projective resolutions
and diagonal approximations for groups that arise as the fundamental
groups of some Seifert $3$-manifolds. In that case, the groups
considered have $4$-periodic cohomology and hence their cohomology
groups are not isomorphic to the cohomology groups of the
corresponding manifolds.

In this paper, we consider the problem of the computing the cohomology
rings of the $3$-manifolds known as torus bundles. Since torus bundles
are $K(\pi,1)$ spaces, we compute their cohomology rings by
constructing a finite free resolution of $\zz$ over $\zz G$, where $G$
is the fundamental group of the torus bundle, and then we construct a
(partial) diagonal approximation for that resolution. Finally, we
compute the multiplicative structure of $H^*(G,\zz)$ and
$H^*(G,\zz_p)$, where $p$ is prime. Although we only did the
calculation for the trivial coefficients $\zz$ and $\zz_p$, the
methods we employ also allow us to deal with other systems of
coefficients, even non trivial ones. We didn't further explore the
computation for non trivial coefficients in order to keep the length
of the paper to a reasonable size, leaving it for specific
applications.

We observe also that J. Hillman has recently determined the cohomology
ring of the $Sol^3$-manifolds with coefficients $\zz_2$ using
different methods. See~\cite{Hillman}.

\bigskip
A {\it torus bundle} is the total space $E$ of a fiber bundle with the
torus $S^1\times S^1$ as the fiber and the circle $S^1$ as the base
space. Given a torus bundle $(S^1\times S^1)\to E \to S^1$, the
fundamental group of the total space $E$ fits in an exact sequence of
groups
\[
\xymatrix{
1 \ar[r] & \pi_1(S^1\times S^1) \ar[r] & \pi_1(E) \ar[r] & \pi_1(S^1) \ar[r] & 1\phantom{,} \\
1 \ar[r] & \zz\oplus\zz \ar[r]\ar@{-}[u]|\cong & \pi_1(E) \ar[r]\ar@{=}[u] & \zz \ar[r]\ar@{-}[u]|\cong & 1,
}
\]
and that implies $\pi_1(E) \cong (\zz\oplus\zz)\rtimes_\theta \zz$ for
some action $\theta\colon \zz \to \Aut(\zz\oplus\zz)\cong GL_2(\zz)$,
since $\zz$ is a free group.

Let $\zz\oplus\zz = \gen{a,b \mid ab=ba}$ be the fundamental group of
the torus and $\zz = \gen{t}$ be the fundamental group of the circle,
and let the action $\theta$ be such that $\theta(t)(a) = a^\alpha
b^\beta$ and $\theta(t)(b) = a^\gamma b^\delta$ for
$\begin{bmatrix}\alpha & \gamma \\ \beta & \delta\end{bmatrix} \in
  GL_2(\zz)$. For the elements of $G=\pi_1(E)$, we use the notation $G
  = \{(a^mb^n,t^k) \::\: m,n,k\in \zz\}$, and we write $\theta
  = \begin{bmatrix}\alpha & \gamma \\ \beta & \delta\end{bmatrix}$.

In Section~\ref{section:resolution}, we construct a finite free
resolution $P$ of $\zz$ as a trivial $\zz G$-module and compute the
groups $H^*(G,\zz)$ and $H^*(G,\zz_p)$ for $p$ prime.  In
Section~\ref{section:diagonal}, we define a partial diagonal
approximation $\Delta\colon P\to P\otimes P$, which will be enough to
allow us to compute the cup product in $H^*(G,\zz)$ and
$H^*(G,\zz_p)$.

\section{Free resolution of $\zz$ over $\zz G$}\label{section:resolution}

Given the exact sequence of groups
\[
\xymatrix{
1 \ar[r] & K \ar[r] & G \ar[r] & H \ar[r] & 1,
}
\]
Wall describes in~\cite{Wall61} how we can construct a free resolution
of $\zz$ over $\zz G$ from free resolutions of $\zz$ over $\zz K$ and
$\zz H$: let
\[
\xymatrix{
\cdots \ar[r] & B_r \ar[r] & \cdots \ar[r] & B_1 \ar[r] &
  B_0 \ar[r]^-\varepsilon & \zz \ar[r] & 0
}
\]
be a free resolution over $\zz K$, and let
\[
\xymatrix{
\cdots \ar[r] & C_s \ar[r] & \cdots \ar[r] & C_1 \ar[r] &
  C_0 \ar[r] & \zz \ar[r] & 0
}
\]
be a free resolution over $\zz H$, with $C_s$ free on $\alpha_s$
generators.

Since $\zz G$ is a free $\zz
K$-module, $\zz G\otimes_{\zz K} B_r$ is a free $\zz G$-module and
$1\otimes\varepsilon\colon \zz G\otimes_{\zz K} B \to \zz G
\otimes_{\zz K} \zz \cong \zz H$ induces an isomorphism in homology.
We define $D_s$ as the direct sum of $\alpha_s$ copies of $\zz
G\otimes_{\zz K} B$, which gives us an augmentation map over the
direct sum of $\alpha_s$ copies of $\zz H$, that we identify with
$C_s$ and write $\varepsilon_s\colon D_s \to C_s$. If $A_{r,s}$ is the
submodule of $D_s$ given by the direct sum of $\alpha_s$ copies of
$\zz G\otimes_{\zz K} B_r$, then $A_{r,s}$ is a free $\zz G$-module
and $D_s$ is the direct sum of the modules $A_{r,s}$. Denoting by
$d_0$ the differential in each of the complexes $D_s$, Wall proves
in~\cite{Wall61} the following lemma:

\begin{lemma}[Wall]\label{lemma:1wall}
There are $\zz G$-homomorphisms $d_k\colon A_{r,s} \to A_{r+k-1,s-k}$
such that
\begin{itemize}
\item[(i)] $d_1\varepsilon_{s-1} = \varepsilon_sd\colon A_{0,s} \to
  C_{s-1}$, where $d$ denotes the differential in $C$;
\item[(ii)] $\ds \sum_{i=0}^kd_id_{k-i} = 0$ for all $k$, where $d_k$
  is null if $r=k=0$ or if $s<k$.
\end{itemize}
\end{lemma}

Given those homomorphisms, the following theorem is then proved:

\begin{theorem}[Wall]\label{theorem:1wall}
Let $A$ be the direct sum of the modules $A_{r,s}$, graded by $\dim
A_{r,s} = r+s$, and let $d = \sum d_k$. The complex $(A,d)$ is
acyclic, and hence gives us a free resolution of $\zz$ over $\zz G$.
\end{theorem}

\bigskip
Let's apply these ideas for the group extension
\[
\xymatrix{
1 \ar[r] & \zz\oplus\zz \ar[r] & G \ar[r] & \zz \ar[r] & 1,
}
\]
where $G = \pi_1(E)$ is the group given in the Introduction. A free
resolution of $\zz$ over $\zz[\zz\oplus\zz]$ is given by

\begin{equation}\label{eq:resZZ}
\xymatrix{
0 \ar[r] & Q_2 \ar[r]^-{d_2} & Q_1 \ar[r]^-{d_1} & Q_0 \ar[r]^-{\varepsilon} & \zz \ar[r] & 0,
}
\end{equation}
where
\begin{align*}
Q_0 &= \gen{e_0} \cong \zz[\zz\oplus\zz], \\
Q_1 &= \gen{e_1^1,e_1^2} \cong \zz[\zz\oplus\zz]\oplus \zz[\zz\oplus\zz], \\
Q_2 &= \gen{e_2} \cong \zz[\zz\oplus\zz],
\end{align*}
and the differentials $d_i$ ($i=1,2$) and $\varepsilon$ are given by
\begin{equation}\label{eq:homsZZ}
\begin{array}{rl}
\varepsilon(e_0) &\!\!\!\!= 1, \\
d_1(e_1^1) &\!\!\!\!= (a-1)e_0, \\
d_1(e_1^2) &\!\!\!\!= (b-1)e_0, \\
d_2(e_2) &\!\!\!\!= (1-b)e_1^1 + (a-1)e_1^2.
\end{array}
\end{equation}

The fact that~(\ref{eq:resZZ}) is indeed a free resolution can be
easily seen considering the cell decomposition of the universal cover
$\rr^2$ of the torus. Also, a free resolution of $\zz$ over $\zz\zz$
is given by
\begin{equation}\label{eq:resZ}
\xymatrix{
0 \ar[r] & \zz\zz \ar[r]^-{t-1} & \zz\zz \ar[r]^-{\varepsilon} & \zz \ar[r] & 0,
} 
\end{equation}
as is shown, for example, in~\cite{Brown82}. If we can find the
homomorphisms of Lemma~\ref{lemma:1wall}, then we'll have our desired
free resolution of $\zz$ over $\zz G$.

From now on, let
\[
\begin{bmatrix}
m_1 & m_2 \\
n_1 & n_2
\end{bmatrix} = 
-\theta^{-1} = 
\frac{1}{\alpha\delta - \beta\gamma}
\begin{bmatrix}
-\delta & \gamma \\
\beta & -\alpha
\end{bmatrix},
\]
and let $A$, $B$, $C$, $D$, $E\in \zz G$ be the elements given by
\begin{align}\label{eq:A}
A &= \begin{cases}
(1,1) + \ds\sum_{k=1}^{m_1} (a^{-k\alpha}b^{-k\beta},t), & \text{ if $m_1 > 0$,}\\
(1,1) - \ds\sum_{k=0}^{-m_1-1}(a^{k\alpha}b^{k\beta},t), & \text{ if $m_1 < 0$,}\\
(1,1), & \text{ if $m_1 = 0$,}
\end{cases}
\end{align}

\begin{align}\label{eq:B}
B &= \begin{cases}
\ds\sum_{k=0}^{n_1-1} (a^{1+k\gamma}b^{k\delta},t), & \text{ if $n_1 > 0$,}\\
-\ds\sum_{k=1}^{-n_1}(a^{1-k\gamma}b^{-k\delta},t), & \text{ if $n_1 < 0$,}\\
0, & \text{ if $n_1 = 0$,}
\end{cases}
\end{align}

\begin{align}\label{eq:C}
C &= \begin{cases}
\ds\sum_{k=1}^{m_2} (a^{-k\alpha}b^{-k\beta},t), & \text{ if $m_2 > 0$,}\\
-\ds\sum_{k=0}^{-m_2-1}(a^{k\alpha}b^{k\beta},t), & \text{ if $m_2 < 0$,}\\
0, & \text{ if $m_2 = 0$,}
\end{cases}
\end{align}

\begin{align}\label{eq:D}
D &= \begin{cases}
(1,1) + \ds\sum_{k=0}^{n_2-1} (a^{k\gamma}b^{1+k\delta},t), & \text{ if $n_2 > 0$,}\\
(1,1) - \ds\sum_{k=1}^{-n_2}(a^{-k\gamma}b^{1-k\delta},t), & \text{ if $n_2 < 0$,}\\
(1,1), & \text{ if $n_2=0$,}
\end{cases}
\end{align}

\begin{equation}\label{eq:E}
E = -(1,1) + \sum_{(m,n)\in I_1\times J_1} (a^{\alpha m + \gamma n}b^{\beta m + \delta n},t) - \sum_{(m,n)\in I_2\times J_2} (a^{\alpha m + \gamma n}b^{\beta m + \delta n},t),
\end{equation}
where the sets $I_1$, $I_2$, $J_1$, $J_2$ vary depending on the signs
of $m_1$, $n_1$, $m_2$ and $n_2$ and are completely described in
Table~\ref{table:ij}.

\begin{theorem}\label{theorem:resolution}
Let
\begin{align*}
P_0 &= \gen{x} \cong \zz G,\\
P_1 &= \gen{y_1}\oplus\gen{y_2}\oplus\gen{y_3} \cong \zz G \oplus \zz G \oplus \zz G,\\
P_2 &= \gen{z_1}\oplus\gen{z_2}\oplus\gen{z_3} \cong \zz G \oplus \zz G \oplus \zz G,\\
P_3 &= \gen{w} \cong \zz G.
\end{align*}
The sequence
\begin{equation}\label{eq:resGZ}
\xymatrix{
0 \ar[r] & P_3 \ar[r]^-{\partial_3} & P_2 \ar[r]^-{\partial_2} & P_1 \ar[r]^-{\partial_1} & P_0 \ar[r]^-{\varepsilon_0} & {\zz} \ar[r] & 0
}
\end{equation}
is a free resolution of $\zz$ over $\zz G$, where the maps
$\partial_i$ ($i=1,2,3$) and $\varepsilon_0$ are defined by
\[
\begin{array}{l}
\varepsilon_0(x) = 1,\\
\\
\partial_1(y_1) = [(a,1)-(1,1)]x,\\
\partial_1(y_2) = [(b,1)-(1,1)]x,\\
\partial_1(y_3) = [(1,t)-(1,1)]x,\\
\\
\partial_2(z_1) = Ay_1 + By_2 + [(a,1)-(1,1)]y_3, \\
\partial_2(z_2) = Cy_1 + Dy_2 + [(b,1)-(1,1)]y_3, \\
\partial_2(z_3) = [(1,1)-(b,1)]y_1 + [(a,1)-(1,1)]y_2, \\
\\
\partial_3(w) = [(1,1)-(b,1)]z_1 + [(a,1)-(1,1)]z_2 + Ez_3.
\end{array}
\]
\end{theorem}

\begin{proof}
Applying the ideas of Wall~\cite{Wall61}, the modules $A_{r,s}$ are such that $A_{r,s}\ne 0$ only for $0\le r\le 2$ and $0\le s \le 1$, and
\[
\begin{array}{l}
A_{2,1} \cong A_{2,0} \cong A_{0,1} \cong A_{0,0} \cong \zz G,\\
A_{1,1} \cong A_{1,0} \cong \zz G\oplus \zz G.
\end{array}
\]
Using Lemma~\ref{lemma:1wall}, we get the diagram
\begin{equation}\label{eq:diagramwall}
\xymatrix{
& \gen{w}\ar@{=}[d] & \gen{z_1}\oplus\gen{z_2} \ar@{=}[d] & \gen{y_3} \ar@{=}[d] & \\
0 \ar[r] &\zz G \ar[r]^-{d_0}\ar[d]^-{d_1} & \zz G\oplus\zz G \ar[r]^-{d_0}\ar[d]^-{d_1} & \zz G \ar[r]^-{\varepsilon_1}\ar[d]^-{d_1} & \zz\zz\phantom{,} 
\ar[d]^-{t-1}\\
0 \ar[r] &\zz G \ar[r]^-{d_0}\ar@{=}[d] & \zz G\oplus\zz G \ar[r]^-{d_0}\ar@{=}[d] & \zz G \ar[r]^-{\varepsilon_0}\ar@{=}[d] & \zz\zz,\\
& \gen{z_3} & \gen{y_1}\oplus\gen{y_2} & \gen{x} &
}
\end{equation}
in which we must determine the maps $\varepsilon_0$, $\varepsilon_1$,
$d_0$ and $d_1$. Having done that, Theorem~\ref{theorem:1wall} will
give us the desired resolution. The homomorphisms $\varepsilon_0$,
$\varepsilon_1$ and $d_0$ are easily calculated:
\begin{align*}
\varepsilon_0 \colon &\gen{x} \to \zz\zz\\
&\varepsilon_0((a^mb^n, t^k)x) = t^k,\\
\varepsilon_1 \colon &\gen{y_3} \to \zz\zz\\
&\varepsilon_1((a^mb^n, t^k)y_3) = t^k,\\
d_0 \colon &\gen{w} \to \gen{z_1}\oplus\gen{z_2}\\
&d_0(w) = ((1,1)-(b,1))z_1 + ((a,1)-(1,1))z_2,\\
d_0 \colon &\gen{z_3} \to \gen{y_1}\oplus\gen{y_2}\\
&d_0(z_3) = ((1,1)-(b,1))y_1 + ((a,1)-(1,1))y_2,\\
d_0 \colon &\gen{z_1}\oplus\gen{z_2} \to \gen{y_3}\\
&d_0(z_1) = ((a,1)-(1,1))y_3,\\
&d_0(z_2) = ((b,1)-(1,1))y_3,\\
d_0 \colon &\gen{y_1}\oplus\gen{y_2} \to \gen{x}\\
&d_0(y_1) = ((a,1)-(1,1))x,\\
&d_0(y_2) = ((b,1)-(1,1))x.\\
\end{align*}
The three maps $d_1$ are computed as follows: the map
$d_1:\gen{y_3}\to \gen{x}$ must satisfy $\varepsilon_0\circ d_1 =
(t-1)\circ\varepsilon_1$, and that happens if we define $d_1(y_3) =
\bigl((1,t)-(1,1)\bigr)x$. For the map $d_1 \colon
\gen{z_1}\oplus\gen{z_2} \to \gen{y_1}\oplus\gen{y_2}$, consider the
diagram
\[
\xymatrix{
\gen{z_1}\oplus\gen{z_2} \ar[r]^-{d_0}\ar[d]_-{d_1} & \gen{y_3}\phantom{,} \ar[d]^-{d_1} \\
\gen{y_1}\oplus\gen{y_2} \ar[r]_-{d_0} & \gen{x},
}
\]
where $d_0d_1+d_1d_0 = 0$. If we write $d_1(z_1) = Ay_1+By_2$ and
$d_1(z_2) = Cy_1+Dy_2$ for some $A$, $B$, $C$, $D\in \zz G$, then
\begin{equation}\label{eq:AB}
(d_0d_1+d_1d_0)(z_1) = 0 \iff A((a,1)-(1,1)) + B((b,1)-(1,1)) = (1,t)-(1,1)+(a,1)-(a,t).
\end{equation}
In order to solve~(\ref{eq:AB}), we write
\begin{align*}
A &= (1,1) + \ds\sum_{m,n\in\zz} g_{mn}(a^mb^n,t),\\
B &= \ds\sum_{m,n\in\zz} h_{mn}(a^mb^n,t),
\end{align*}
so~(\ref{eq:AB}) becomes
\begin{equation}\label{eq:AB2}
\ds\sum_{m,n\in\zz} \left(g_{(m-\alpha)(n-\beta)} - g_{mn} + h_{(m-\gamma)(n-\delta)} - h_{mn}\right) (a^mb^n,t) = (1,t)-(a,t).
\end{equation}
We define a directed graph to solve~(\ref{eq:AB2}) in the following
way: the vertices of the graph are the points of $\zz\times\zz$, and
each vertex $(m,n)$ is the origin of two edges, one going to
$(m-\alpha,n-\beta)$ and the other to $(m-\gamma,n-\delta)$. Given
that $\det\theta=\pm 1$, we can draw the graph as a square grid graph
with horizontal edges going left and vertical edges going down,
where each horizontal edge goes from $(m,n)$ to $(m-\alpha,n-\beta)$
and each vertical edge goes from $(m,n)$ to $(m-\gamma, n-\delta)$. On
each vertex $(m,n)$, we put two labels (yet to be determined): one
with the value of $g_{mn}$ and another with the value of
$h_{mn}$. Then, on each horizontal edge we put a label with the value
of $g_{(m-\alpha)(n-\beta)} - g_{mn}$, and on each vertical edge we
put a label with the value of $h_{(m-\gamma)(n-\delta)} -
h_{mn}$. Observing that we can write~(\ref{eq:AB2}) as
\begin{equation}\label{eq:AB3}
g_{(m-\alpha)(n-\beta)} - g_{mn} + h_{(m-\gamma)(n-\delta)} - h_{mn} =
\begin{cases}
1, & \text{if $(m,n)=(0,0)$,}\\
-1, & \text{if $(m,n) = (1,0)$,}\\
0, & \text{otherwise,}
\end{cases}
\end{equation}
we use the graph we just defined (except for labels) to solve it in
the following way: first, we find a path (in the undirected graph
obtained from the one we defined) between the vertices $(0,0)$ and
$(1,0)$. In order to do that, let $m_1$ and $n_1$ be integers such
that
\[
\left|
\begin{array}{l}
-\alpha m_1 - \gamma n_1 = 1\\
-\beta m_1 - \delta n_1 = 0.
\end{array}
\right.
\]
Hence we can go from $(0,0)$ to $(1,0)$ passing through $|m_1|$
horizontal edges and followed by $|n_1|$ vertical edges. Just to give
a concrete case, suppose $m_1<0$ and $n_1>0$. Then our path looks like
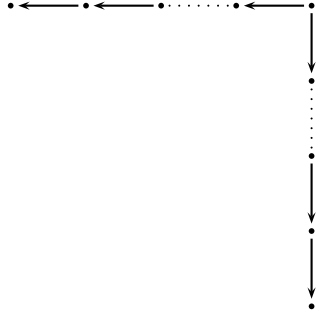
\begin{center}
\psset{unit=1cm}
\begin{pspicture}(-0.5,-0.5)(4.5,4.5)
\psdots[dotscale=0.6](0,4)(1,4)(2,4)(3,4)(4,4)(4,0)(4,1)(4,2)(4,3)
\psline{->}(0.9,4)(0.1,4)
\psline{->}(1.9,4)(1.1,4)
\psline{->}(3.9,4)(3.1,4)
\psline{->}(4,3.9)(4,3.1)
\psline{->}(4,1.9)(4,1.1)
\psline{->}(4,0.9)(4,0.1)
\psline[linestyle=dotted](2.9,4)(2.1,4)
\psline[linestyle=dotted](4,2.9)(4,2.1)
\end{pspicture}
\end{center}
where the top left vertex is $(0,0)$ and the bottom right vertex is
$(1,0)$. From~(\ref{eq:AB3}), we know that the sum of the labes of the
edges originated at $(0,0)$ must be $1$, the sum of the labels of the
edges originated at $(1,0)$ must be $-1$, and the sum of the labels
originated from any other vertex must be zero. This can be
accomplished with the following labels on the vertices:
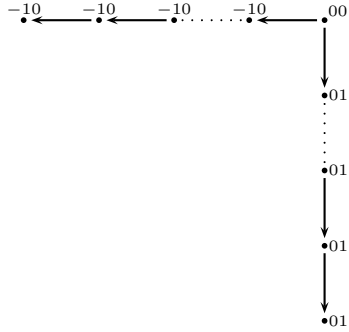
\begin{center}
\psset{unit=1cm}
\begin{pspicture}(-0.5,-0.5)(4.5,4.5)
\psdots[dotscale=0.6](0,4)(1,4)(2,4)(3,4)(4,4)(4,0)(4,1)(4,2)(4,3)
\psline{->}(0.9,4)(0.1,4)
\psline{->}(1.9,4)(1.1,4)
\psline{->}(3.9,4)(3.1,4)
\psline{->}(4,3.9)(4,3.1)
\psline{->}(4,1.9)(4,1.1)
\psline{->}(4,0.9)(4,0.1)
\psline[linestyle=dotted](2.9,4)(2.1,4)
\psline[linestyle=dotted](4,2.9)(4,2.1)
\uput{0.05cm}[90](0,4){$\scriptstyle -10$}
\uput{0.05cm}[90](1,4){$\scriptstyle -10$}
\uput{0.05cm}[90](2,4){$\scriptstyle -10$}
\uput{0.05cm}[90](3,4){$\scriptstyle -10$}
\uput{0.07cm}[45](4,4){$\scriptstyle 00$}
\uput{0.06cm}[0](4,0){$\scriptstyle 01$}
\uput{0.06cm}[0](4,1){$\scriptstyle 01$}
\uput{0.06cm}[0](4,2){$\scriptstyle 01$}
\uput{0.06cm}[0](4,3){$\scriptstyle 01$}
\end{pspicture}
\end{center}
For each vertex, the left label is $g_{mn}$ and the right one is
$h_{mn}$. For all the vertices that were not drawn, we have $g_{mn} =
h_{mn} = 0$. From the labels above we read the values of $A$ and $B$:
\begin{align*}
A &= (1,1) - \ds\sum_{k=0}^{m_1-1} (a^{k\alpha}b^{k\beta},t),\\
B &= \ds\sum_{k=0}^{n_1-1}(a^{1+k\gamma}b^{k\delta},t).
\end{align*}
Given any $\theta$, we can compute explicitly the values of $m_1$,
$n_1$, and then discover the values of $A$ and $B$. The same can be
done with the elements $C$ and $D$, and the values we find out for
these four variables are the ones given in the statement of the
theorem.

Finally, it remains to compute the map $d_1 \colon \gen{w} \to
\gen{z_3}$ in the diagram
\[
\xymatrix{
\gen{w} \ar[r]^-{d_0}\ar[d]_-{d_1} &
\gen{z_1}\oplus\gen{z_2} \ar[d]^-{d_1}\\
\gen{z_3} \ar[r]_-{d_0} & \gen{y_1}\oplus\gen{y_2}.
}
\]
If $d_1(w) = Ez_3$ for some $E\in \zz G$, then $(d_0d_1 + d_1d_0)(w) =
0$ is equivalent to
\begin{equation}\label{eq:E0}
\left|
\begin{array}{l}
E((1,1)-(b,1)) = ((b,1)-(1,1))A + ((1,1)-(a,1))C\\
E((a,1)-(1,1)) = ((b,1)-(1,1))B + ((1,1)-(a,1))D
\end{array}
\right..
\end{equation}
Given the expressions, for $A$, $B$, $C$ and $D$, we guess that $E$
can be written in the form
\[
E = -(1,1) + \sum_{m,n\in\zz} h_{mn}(a^mb^n,t)
\]
for some integers $h_{mn}$ yet to be determined. Now the trick to
compute the integers $h_{mn}$ is similar to the one we used to find
the elements $A$, $B$, $C$ and $D$. In order to show the computations
in an actual case, let's suppose that $m_1<0$, $n_1>0$, $m_2>0$ and
$n_2<0$ (this case arises when $\alpha$, $\beta$, $\gamma$ and
$\delta$ are positive and $\det\theta=1$). In this case we can
write~(\ref{eq:E0}) as
\begin{gather}
\ds ((b,1)-(1,1)) + \sum_{m,n\in\zz}(h_{mn}-h_{(m-\gamma)(n-\delta)})(a^mb^n,t) = ((b,1)-(1,1)) + \notag\\
+\sum_{k=0}^{-m_1-1}\left[(a^{k\alpha}b^{k\beta},t)-(a^{k\alpha}b^{1+k\beta},t)\right] + \sum_{k=1}^{m_2}\left[(a^{-k\alpha}b^{-k\beta},t) - (a^{1-k\alpha}b^{-k\beta},t)\right] \label{eq:E1}
\end{gather}
and
\begin{gather}
\ds ((1,1)-(a,1)) + \sum_{m,n\in\zz}(h_{(m-\alpha)(n-\beta)}-h_{mn})(a^mb^n,t) = ((1,1)-(a,1)) + \notag\\
+\sum_{k=0}^{n_1-1}\left[(a^{1+k\gamma}b^{1+k\delta},t)-(a^{1+k\gamma}b^{k\delta},t)\right] - \sum_{k=1}^{-n_2}\left[(a^{-k\gamma}b^{1-k\delta},t)-(a^{1-k\gamma}b^{1-k\delta},t)\right]. \label{eq:E2}
\end{gather}
Once again we construct a directed graph with the points of
$\zz\times\zz$ as vertices. Each vertex $(m,n)$ is the origin of two
edges, one going to $(m-\alpha,n-\beta)$ and another going to
$(m+\gamma, n+\delta)$. On the edge with origin at $(m,n)$ and going
to $(m-\alpha,n-\beta)$, we put a label with the value of
$h_{(m-\alpha)(n-\beta)}-h_{mn}$, and on the edge with origin at
$(m-\gamma,n-\delta)$ and going to $(m,n)$, we put a label with the
value of $h_{mn}-h_{(m-\gamma)(n-\delta)}$. Those values are available
to us and are given by the equations~(\ref{eq:E1})
and~(\ref{eq:E2}). Our task then consists in putting a label with the
value of $h_{mn}$ on each vertex $(m,n)$ in such a way that it is
consistent with the labels of the edges of the graph. In order to do
so, it is more convenient to draw the graph with horizontal and
vertical edges, like we did before, which can be done applying the
linear map $T\colon \zz\times\zz \to \zz\times\zz$ given by
\[
T(m,n) = \theta^{-1}\begin{bmatrix}
m\\n
\end{bmatrix}
=
-\begin{bmatrix}
m_1 & m_2\\
n_1 & n_2
\end{bmatrix}
\begin{bmatrix}
m\\n
\end{bmatrix} =
(-m_1m-m_2n, -n_1m-n_2n).
\]
to the vertices of our graph. Then, in the case we are considering,
our graph looks like this:

\bigskip
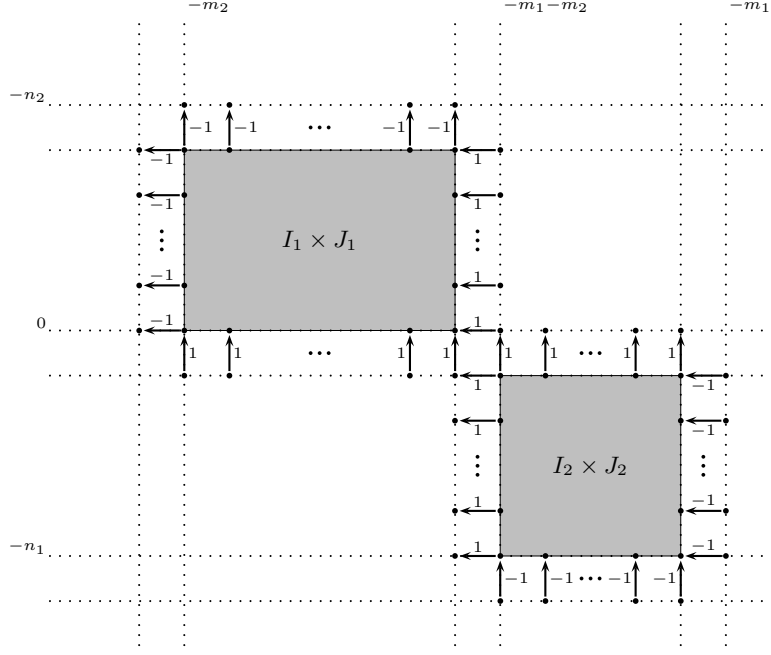
\begin{center}
\psset{unit=0.6cm}
\begin{pspicture}(0,0)(16,14)
\psframe[linewidth=0, fillstyle=solid, fillcolor=lightgray](3,7)(9,11)
\psframe[linewidth=0, fillstyle=solid, fillcolor=lightgray](10,2)(14,6)
\rput(6,9){$I_1\times J_1$}
\rput(12,4){$I_2\times J_2$}
\psdots[dotscale=0.6](2,7)(2,8)(2,10)(2,11)
\psdots[dotscale=0.6](3,6)(3,7)(3,8)(3,10)(3,11)(3,12)
\psdots[dotscale=0.6](4,6)(4,7)(4,11)(4,12)
\psdots[dotscale=0.6](8,6)(8,7)(8,11)(8,12)
\psdots[dotscale=0.6](9,2)(9,3)(9,5)(9,6)(9,7)(9,8)(9,10)(9,11)(9,12)
\psdots[dotscale=0.6](10,1)(10,2)(10,3)(10,5)(10,6)(10,7)(10,8)(10,10)(10,11)
\psdots[dotscale=0.6](11,1)(11,2)(11,6)(11,7)
\psdots[dotscale=0.6](13,1)(13,2)(13,6)(13,7)
\psdots[dotscale=0.6](14,1)(14,2)(14,3)(14,5)(14,6)(14,7)
\psdots[dotscale=0.6](15,2)(15,3)(15,5)(15,6)
\psset{linestyle=dotted}
\psline(2,0)(2,14)
\psline(3,0)(3,14) \uput{0.05cm}[45](3,14){$\scriptstyle -m_2$}
\psline(9,0)(9,14)
\psline(10,0)(10,14) \uput{0.05cm}[45](10,14){$\scriptstyle -m_1-m_2$}
\psline(14,0)(14,14)
\psline(15,0)(15,14) \uput{0.05cm}[45](15,14){$\scriptstyle -m_1$}
\psline(0,1)(16,1)
\psline(0,2)(16,2) \uput{0.05cm}[135](0,2){$\scriptstyle -n_1$}
\psline(0,6)(16,6)
\psline(0,7)(16,7) \uput{0.05cm}[135](0,7){$\scriptstyle 0$}
\psline(0,11)(16,11)
\psline(0,12)(16,12) \uput{0.05cm}[135](0,12){$\scriptstyle -n_2$}
\psset{linestyle=solid}
\psline{->}(2.9,7)(2.1,7) \uput{0.05cm}[90](2.5,7){$\scriptstyle -1$}
\psline{->}(2.9,8)(2.1,8) \uput{0.05cm}[90](2.5,8){$\scriptstyle -1$}
\psline{->}(2.9,10)(2.1,10) \uput{0.05cm}[-90](2.5,10){$\scriptstyle -1$}
\psline{->}(2.9,11)(2.1,11) \uput{0.05cm}[-90](2.5,11){$\scriptstyle -1$}
\psdots[dotscale=0.4](2.5,8.8)(2.5,9)(2.5,9.2)
\psline{->}(3,11.1)(3,11.9) \uput{0.05cm}[0](3,11.5){$\scriptstyle -1$}
\psline{->}(4,11.1)(4,11.9) \uput{0.05cm}[0](4,11.5){$\scriptstyle -1$}
\psline{->}(8,11.1)(8,11.9) \uput{0.05cm}[180](8,11.5){$\scriptstyle -1$}
\psline{->}(9,11.1)(9,11.9) \uput{0.05cm}[180](9,11.5){$\scriptstyle -1$}
\psdots[dotscale=0.4](5.8,11.5)(6,11.5)(6.2,11.5)
\psline{->}(3,6.1)(3,6.9) \uput{0.05cm}[0](3,6.5){$\scriptstyle 1$}
\psline{->}(4,6.1)(4,6.9) \uput{0.05cm}[0](4,6.5){$\scriptstyle 1$}
\psline{->}(8,6.1)(8,6.9) \uput{0.05cm}[180](8,6.5){$\scriptstyle 1$}
\psline{->}(9,6.1)(9,6.9) \uput{0.05cm}[180](9,6.5){$\scriptstyle 1$}
\psdots[dotscale=0.4](5.8,6.5)(6,6.5)(6.2,6.5)
\psline{->}(10,6.1)(10,6.9) \uput{0.05cm}[0](10,6.5){$\scriptstyle 1$}
\psline{->}(11,6.1)(11,6.9) \uput{0.05cm}[0](11,6.5){$\scriptstyle 1$}
\psline{->}(13,6.1)(13,6.9) \uput{0.05cm}[180](13,6.5){$\scriptstyle 1$}
\psline{->}(14,6.1)(14,6.9) \uput{0.05cm}[180](14,6.5){$\scriptstyle 1$}
\psdots[dotscale=0.4](11.8,6.5)(12,6.5)(12.2,6.5)
\psline{->}(9.9,11)(9.1,11) \uput{0.05cm}[-90](9.5,11){$\scriptstyle 1$}
\psline{->}(9.9,10)(9.1,10) \uput{0.05cm}[-90](9.5,10){$\scriptstyle 1$}
\psline{->}(9.9,8)(9.1,8) \uput{0.05cm}[90](9.5,8){$\scriptstyle 1$}
\psline{->}(9.9,7)(9.1,7) \uput{0.05cm}[90](9.5,7){$\scriptstyle 1$}
\psdots[dotscale=0.4](9.5,8.8)(9.5,9)(9.5,9.2)
\psline{->}(9.9,6)(9.1,6) \uput{0.05cm}[-90](9.5,6){$\scriptstyle 1$}
\psline{->}(9.9,5)(9.1,5) \uput{0.05cm}[-90](9.5,5){$\scriptstyle 1$}
\psline{->}(9.9,3)(9.1,3) \uput{0.05cm}[90](9.5,3){$\scriptstyle 1$}
\psline{->}(9.9,2)(9.1,2) \uput{0.05cm}[90](9.5,2){$\scriptstyle 1$}
\psdots[dotscale=0.4](9.5,3.8)(9.5,4)(9.5,4.2)
\psline{->}(14.9,6)(14.1,6) \uput{0.05cm}[-90](14.5,6){$\scriptstyle -1$}
\psline{->}(14.9,5)(14.1,5) \uput{0.05cm}[-90](14.5,5){$\scriptstyle -1$}
\psline{->}(14.9,3)(14.1,3) \uput{0.05cm}[90](14.5,3){$\scriptstyle -1$}
\psline{->}(14.9,2)(14.1,2) \uput{0.05cm}[90](14.5,2){$\scriptstyle -1$}
\psdots[dotscale=0.4](14.5,3.8)(14.5,4)(14.5,4.2)
\psline{->}(10,1.1)(10,1.9) \uput{0.05cm}[0](10,1.5){$\scriptstyle -1$}
\psline{->}(11,1.1)(11,1.9) \uput{0.05cm}[0](11,1.5){$\scriptstyle -1$}
\psline{->}(13,1.1)(13,1.9) \uput{0.05cm}[180](13,1.5){$\scriptstyle -1$}
\psline{->}(14,1.1)(14,1.9) \uput{0.05cm}[180](14,1.5){$\scriptstyle -1$}
\psdots[dotscale=0.4](11.8,1.5)(12,1.5)(12.2,1.5)
\end{pspicture}
\end{center}
In the picture above, we did not draw edges with null labels. From the
graph above, define the sets $I_1$, $J_1$, $I_2$, $J_2\subset \zz$ by
\begin{align*}
I_1 &= [-m_2, -m_1-m_2-1]\cap\zz,\\
J_1 &= [0,-n_2-1]\cap\zz,\\
I_2 &= [-m_1-m_2,-m_1-1]\cap\zz,\\
J_2 &= [-n_1,-1]\cap\zz.
\end{align*}
If we put the label $1$ on the vertices of $I_1\times J_1$, the label
$-1$ on the vertices $(m,n)\in I_2\times J_2$ and the label $0$
otherwise, then those labels will be consistent with the ones on the
edges. Applying now the map $T^{-1}$, we obtain
\begin{equation*}
E = -(1,1) + \sum_{(m,n)\in I_1\times J_1} (a^{\alpha m + \gamma n}b^{\beta m + \delta n},t) - \sum_{(m,n)\in I_2\times J_2} (a^{\alpha m + \gamma n}b^{\beta m + \delta n},t).
\end{equation*}
The same construction can be done for all the other cases, and in all
of them the element $E$ is given by~(\ref{eq:E}), for varying sets
$I_1$, $J_1$, $I_2$, $J_2$. A complete description of those sets is
given in Table~\ref{table:ij}.  Now that we have computed all the maps
in the diagram~(\ref{eq:diagramwall}), Theorem~\ref{theorem:1wall}
gives the free resolution of $\zz$ over $\zz G$.
\end{proof}

Given the previous theorem, we now proceed to compute the cohomology
groups $H^*(G,\zz)$ and $H^*(G,\zz_p)$ for $p$ prime.

\begin{theorem}\label{theorem:HZ}
The cohomology groups $H^*(G,\zz)$, where $\zz$ is the trivial $\zz
G$-module, are given by
\begin{align*}
H^0(G,\zz) &\cong \zz,\\
H^1(G,\zz) &\cong (\zz)^{3-\rank(\theta-I)},\\
H^2(G,\zz) &\cong \begin{cases}
\zz\oplus\zz\oplus\zz, & \text{ if $\rank(\theta-I)=0$,} \\
\zz_{\gcd(\beta,\gamma)}\oplus\zz\oplus\zz, & \text{ if $\rank(\theta-I)=1$ and $\det\theta=1$,} \\
\zz_{\gcd(\beta,\gamma,2)}\oplus\zz, & \text{ if $\rank(\theta-I)=1$ and $\det\theta=-1$,} \\
\zz_{c_1}\oplus\zz_{c_2}\oplus\zz, & \text{ if $\rank(\theta-I)=2$ and $\det\theta=1$,} \\
\zz_{c_1}\oplus\zz_{c_2}, & \text{ if $\rank(\theta-I)=2$ and $\det\theta=-1$,}
\end{cases} \\
H^3(G, \zz) &\cong
\begin{cases}
\zz, & \text{ if $\det\theta=1$},\\
\zz_2, & \text{ if $\det\theta=-1$},
\end{cases}\\
H^n(G,\zz) &\cong 0, \text{ if $n\ge 4$}.
\end{align*}
The positive integers $c_1$ and $c_2$ are such that $c_1 \mid c_2$,
$c_1c_2 = |\det(\theta-I)|$.
\end{theorem}

\begin{proof}
Applying the functor $\Hom_{\zz G}(\underline{\phantom{M}}, \zz)$ to
the resolution given in Theorem~\ref{theorem:resolution}, we get
\[
\xymatrix@1@C=16pt{
{\phantom{.}}0\phantom{.} \ar[r] & \Hom_{\zz G}(P_0,\zz) \ar[r]^-{\partial_1^*} & \Hom_{\zz G}(P_1,\zz) \ar[r]^-{\partial_2^*} & \Hom_{\zz G}(P_2,\zz) \ar[r]^-{\partial_3^*} & \Hom_{\zz G}(P_3,\zz) \ar[r] & {\phantom{.}}0\phantom{.} \\
{\phantom{.}}0\phantom{.} \ar@{=}[u] \ar[r] & \zz\ar@{-}[u]|{\cong} \ar[r]^-{\partial_1^*} & \zz\oplus\zz\oplus\zz \ar@{-}[u]|{\cong} \ar[r]^-{\partial_2^*} & \zz\oplus\zz\oplus\zz \ar@{-}[u]|{\cong} \ar[r]^-{\partial_3^*} & \zz \ar@{-}[u]|{\cong} \ar[r] & {\phantom{.}}0. \ar@{=}[u]
}
\]
A quick computation shows that $\partial_1^* = 0$ and $H^0(G,\zz)
\cong \zz$, with a generator being given by $[x^*]$. If
$\varepsilon\colon \zz G\to \zz$ represents the augmentation map, then
the elements $A$, $B$, $C$ and $D$ of Theorem~\ref{theorem:resolution}
given by the equations (\ref{eq:A}), (\ref{eq:B}), (\ref{eq:C}) and
(\ref{eq:D}) are such that $\varepsilon(A) = 1+m_1$, $\varepsilon(B) =
n_1$, $\varepsilon(C) = m_2$ and $\varepsilon(D) = 1+n_2$. Hence the
matrix of $\partial_2^* \colon \zz\oplus\zz\oplus\zz \to
\zz\oplus\zz\oplus\zz$ relative to the dual bases of $\{y_1,y_2,y_3\}$
and $\{z_1,z_2,z_3\}$ is
\begin{equation}\label{eq:partial2*}
{[\partial_2^*]} = \begin{bmatrix}
1+m_1 & n_1 & 0 \\
m_2 & 1+n_2 & 0 \\
0 & 0 & 0
\end{bmatrix}.
\end{equation}
Then $H^1(G,\zz) = \dfrac{\ker \partial_2^*}{\im \partial_1^*} \cong
\ker \partial_2^* \cong (\zz)^{3 - \rank([\partial_2^*])} = (\zz)^{3 -
  \rank(\theta-I)}$, since $\rank(\theta-I) = \rank(I-\theta^{-1})$
and
\[
I-\theta^{-1} = \begin{bmatrix}
1+m_1 & m_2\\
n_1 & 1+n_2
\end{bmatrix}.
\]
We can also exhibit explicit generators for $H^1(G,\zz)$: if
$\rank(\theta-I) = 0$, then $H^1(G,\zz) =
\gen{[y_1^*],[y_2^*],[y_3^*]}$. If $\rank(\theta-I)$, then $H^1(G,\zz)
= \gen{[y_3^*]}$. Finally, if $\rank(\theta-I) = 1$, then one of the
generators of $H^1(G,\zz)$ is $[y_3^*]$, while the other generator is
obtained in the following way: if $(1+m_1) = n_1 = 0$, then the second
generator of $H^1(G,\zz)$ is given by
\[
\left[-\frac{1+n_2}{\gcd(m_2,1+n_2)}y_1^* + \frac{m_2}{\gcd(m_2,1+n_2)}y_2^*\right],
\]
and if $(1+m_1) \ne 0$ or $n_1 \ne 0$ we can take
\[
\left[-\frac{n_1}{\gcd(1+m_1,n_1)}y_1^* + \frac{1+m_1}{\gcd(1+m_1,n_1)}y_2^*\right]
\]
as the second generator of $H^1(G,\zz)$.

The matrix of $\partial_3^*$ relative to the dual bases
of $\{z_1,z_2,z_3\}$ and $\{w\}$ is
\begin{equation}\label{eq:partial3*}
[\partial_3^*] = \begin{bmatrix}
0 & 0 & (-1+\det\theta)
\end{bmatrix},
\end{equation}
since the element $E$ given by equation~(\ref{eq:E}) is such that
$\varepsilon(E) = -1 + |I_1\times J_1| - |I_2\times J_2|$ and the sets
$I_1$, $J_1$, $I_2$ and $J_2$ always satisfy (see
Table~\ref{table:ij})
\begin{equation}\label{eq:i1j1i2j2}
|(I_1\times J_1)| - |(I_2\times J_2)| = \det\theta = \pm 1.
\end{equation}
This implies
\[
H^3(G, \zz) \cong
\begin{cases}
\zz, & \text{ if $\det\theta=1$},\\
\zz_2, & \text{ if $\det\theta=-1$},
\end{cases}
\]
with $[w^*]$ a generator for $H^3(G, \zz)$. Finally we proceed to the
computation of $H^2(G,\zz)$. If $\det\theta=1$, then $z_3^* \in
\ker\partial_3^*$ and
\[
H^2(G, \zz) \cong \frac{\gen{z_1^*}\oplus \gen{z_2^*}}{\im \partial_2^*} \oplus \zz,
\]
whereas if $\det\theta = -1$ we have
\[
H^2(G, \zz) \cong \frac{\gen{z_1^*}\oplus \gen{z_2^*}}{\im \partial_2^*}.
\]
In both cases, the group structure of $\dfrac{\gen{z_1^*}\oplus
  \gen{z_2^*}}{\im \partial_2^*}$ can be obtained calculating the
Smith normal form of $(I-\theta^{-1})$. If $\rank(I-\theta^{-1}) = 0$,
then $\dfrac{\gen{z_1^*}\oplus \gen{z_2^*}}{\im \partial_2^*} =
\gen{[z_1^*]}\oplus\gen{[z_2^*]} \cong \zz\oplus\zz$. If
$\rank(I-\theta^{-1}) = 1$, then $(I-\theta^{-1})$ is a non-zero
matrix that can be written as
\[
\begin{bmatrix}
r & (p/q)r\\
s & (p/q)s
\end{bmatrix} \quad\text{ or }\quad
\begin{bmatrix}
(p/q)r & r \\
(p/q)s & s
\end{bmatrix},
\]
where $r$, $s$, $p$ and $q$ are integers such that $q\ne 0$ and
$\gcd(p,q)=1$. Since both cases are similar, we analyze the first
one. Writing $r=qr'$ and $s=qs'$, we have $I-\theta^{-1} = \begin{bmatrix}
qr' & pr'\\
qs' & ps'
\end{bmatrix}$.
If $k$, $\ell\in\zz$ are such that $pk+q\ell=1$, then
\[
(I-\theta^{-1})\begin{bmatrix}\ell & -p\\k & q \end{bmatrix} =
\begin{bmatrix}r' & 0\\ s' & 0\end{bmatrix},
\]
and the Smith normal form of $(I-\theta^{-1})$ is
$\begin{bmatrix}\gcd(r',s') & 0 \\ 0 & 0\end{bmatrix}$. Hence
  $\dfrac{\gen{z_1^*}\oplus \gen{z_2^*}}{\im \partial_2^*} \cong
  \zz_{\gcd(r',s')}\oplus\zz$. Now, if $p\ne 0$, $\gcd(r',s') =
  \gcd(pr',qs') = \gcd(m_2,n_1) = \gcd(\beta,\gamma)$. If $p=0$ and
  $\det\theta=1$, then $\alpha=\delta=1$ and $\gamma=0$, which implies
  $r'=0$ and $\gcd(r',s') = \gcd(0,\beta) = \beta =
  \gcd(\beta,\gamma)$. And, if $p=0$ and $\det\theta=-1$, then
  $\alpha=-1$, $\delta=1$, $\gamma=0$ and $r'=2$ (for in this case we
  take $q=1$), so $\gcd(r',s') = \gcd(\beta,2) =
  \gcd(\beta,\gamma,2)$. When $\rank(I-\theta^{-1}) = 1$ and
  $\det\theta=-1$, we make the extra observation that $\gcd(r',s') \in
  \{1,2\}$. This can be easily seen, as in this case we have
\[
\begin{bmatrix}
qr' & pr' \\
qs' & ps'
\end{bmatrix} = 
I-\theta^{-1} = \begin{bmatrix}
1+m_1 & m_2\\
n_1 & 1+n_2
\end{bmatrix} =
\begin{bmatrix}
1+\delta & -\gamma \\
-\beta & 1+\alpha
\end{bmatrix},
\]
and $\det(I-\theta^{-1}) = 0 \iff \alpha+\delta=0 \iff (ps'-1) +
(qr'-1) = 0 \iff ps'+qr'=2 \then \gcd(r',s')\mid 2$. Hence, when
$\det\theta=-1$ and $\rank(\theta-I)=1$, we can always write
$\gcd(r',s') = \gcd(\beta,\gamma,2)$.

Finally, if $\rank(I-\theta^{-1})=2$, the Smith normal form of
$(I-\theta^{-1})$ is a matrix $\begin{bmatrix} c_1 & 0\\ 0 & c_2
\end{bmatrix}$,
with $c_1,c_2>0$, $c_1 \mid c_2$ and $c_1c_2 = |\det(I-\theta^{-1})| =
|\det(\theta-I)|$.
\end{proof}

The calculations of the groups $H^*(G,\zz_2)$ and $H^*(G,\zz_p)$ for
and odd prime $p$ are now simple, as the matrices of $\partial_2^*$
and $\partial_3^*$ are obtained from the case of $\zz$ coefficients by
reducing them mod $2$ and mod $p$: for $\zz_2$ coefficients, we
observe that $\partial_3^*=0$, while for $\zz_p$ coefficients we
observe that $\partial_3^*=0$ if $\det\theta=1$ and $\partial_3^*$ is
a bijection if $\det\theta=-1$. We then get the following two
theorems.

\begin{theorem}\label{theorem:HZ2}
The cohomology groups $H^*(G,\zz_2)$ are given by
\begin{align*}
H^0(G, \zz_2) &\cong \zz_2,\\
H^1(G, \zz_2) &\cong (\zz_2)^{3-\rank_{\zz_2}(\theta-I)}, \\
H^2(G, \zz_2) &\cong (\zz_2)^{3-\rank_{\zz_2}(\theta-I)}, \\
H^3(G, \zz_2) &\cong \zz_2,\\
H^n(G, \zz_2) &\cong 0, \text{ if $n\ge 4$}.
\end{align*}
\end{theorem}

\begin{theorem}\label{theorem:HZp}
Let $p$ be an odd prime. The cohomology groups $H^*(G,\zz_p)$, where
$\zz_p$ is the trivial $\zz G$-module, are given by
\begin{align*}
H^0(G, \zz_p) &\cong \zz_p,\\
H^1(G, \zz_p) &\cong (\zz_p)^{3-\rank_{\zz_p}(\theta-I)}, \\
H^2(G, \zz_p) &\cong \begin{cases} (\zz_p)^{3-\rank_{\zz_p}(\theta-I)}, & \text{if $\det\theta=1$}, \\ (\zz_p)^{2-\rank_{\zz_p}(\theta-I)}, & \text{if $\det\theta=-1$}, \end{cases} \\ 
H^3(G, \zz_p) &\cong \begin{cases}\zz_p, & \text{if $\det\theta=1$}, \\ 0, & \text{if $\det\theta=-1$},\end{cases} \\
H^n(G, \zz_p) &\cong 0, \text{ if $n\ge 4$}.
\end{align*}
\end{theorem}

\section{Diagonal approximation and the cup product}\label{section:diagonal}

In order to compute the cup product in the cohomology groups
$H^*(G,\zz)$, $H^*(G,\zz_2)$ and $H^*(G,\zz_p)$, we seek a diagonal
approximation $\Delta\colon P\to (P\otimes P)$ for the free resolution
$P$ given in
Theorem~\ref{theorem:resolution}. In~\cite{TomodaPeter2008}, we find
the following two propositions, which can help us determine $\Delta$.

\begin{proposition}\label{proposition:tomodacontraction}
For a group $G$, let
\[
\xymatrix{
  \cdots \ar[r] & C_n \ar[r] & \cdots \ar[r] & C_1 \ar[r] & C_0 \ar[r]^-{\varepsilon} \ar[r] & \zz \ar[r] & 0
}
\]
be a finitely generated free resolution of $\zz$ over $\zz G$, that
is, each $C_n$ is finitely generated as a $\zz G$-module. If $s$ is a
contracting homotopy for the resolution $C$, then a contracting
homotopy $\tilde{s}$ for the free resolution $C\otimes C$ of $\zz$
over $\zz G$ is given by
\begin{align*}
\tilde{s}_{-1}\colon &\zz \to C_0\otimes C_0 \\
&\tilde{s}_{-1}(1) = s_{-1}(1)\otimes s_{-1}(1), \\
\tilde{s}_n\colon &(C\otimes C)_n \to (C\otimes C)_{n+1} \\
& \tilde{s}_n(u_i\otimes v_{n-i}) = s_i(u_i)\otimes v_{n-i} + s_{-1}\varepsilon(u_i)\otimes s_{n-i}(v_{n-i}), \quad\text{se $n\ge 0$},
\end{align*}
where $s_{-1}\varepsilon\colon C_0\to C_0$ is extended to
$s_{-1}\varepsilon = \{(s_{-1}\varepsilon)_n \colon C_n \to C_n\}$ in
such a way that $(s_{-1}\varepsilon)_n = 0$ for $n\ge 1$.
\end{proposition}

\begin{proposition}\label{proposition:tomodadiagonal}
For a group $G$, let
\[
\xymatrix{
  \cdots \ar[r] & C_n \ar[r]^-{d_n} & \cdots \ar[r] & C_1 \ar[r]^-{d_1} & C_0 \ar[r]^-{\varepsilon} \ar[r] & \zz \ar[r] & 0
}
\]
be a finitely generated free resolution of $\zz$ over $\zz G$ (i.e.,
each $C_n$ is a finitely generated free $\zz G$-module), and let $s$
be a contracting homotopy for this resolution $C$. If $\tilde{s}$ is
the contracting homotopy for the resolution $C\otimes C$ given by
Proposition~\ref{proposition:tomodacontraction}, then a diagonal
approximation $\Delta\colon C\to C\otimes C$ can be defined in the
following way: for each $n\ge 0$, the map $\Delta_n\colon C_n \to
(C\otimes C)_n$ is given in each generator $\rho$ of $C_n$ by
\[
\begin{array}{l}
\Delta_0 = s_{-1}\varepsilon\otimes s_{-1}\varepsilon,\\
\Delta_n(\rho) = \tilde{s}_{n-1}\Delta_{n-1}d_n(\rho), \quad\text{if $n\ge 1$}.
\end{array}
\]
\end{proposition}

The two above propositions tell us that if we can manage to find a
contracting homotopy for the resolution $P$ given in
Theorem~\ref{theorem:resolution}, then we can construct a diagonal
approximation $\Delta$ and then procceed to calculate the cup product
in the cohomology ring. The maps $s_{-1}\colon \zz \to P_0$ and
$s_0\colon P_0\to P_1$ are easy to define: we take $s_{-1}(1) = x$ and
$\ds s_0((a^mb^n,t^k)x) = \frac{\partial (a^m,1)}{\partial (a,1)}y_1 +
(a^m,1)\frac{\partial (b^n,1)}{\partial (b,1)}y_2 +
(a^mb^n,1)\frac{\partial (1,t^k)}{\partial (1,y)}y_3$, where the
partial derivatives are the Fox derivatives, and it is immediate to
check that $\varepsilon s_{-1} = \id_\zz$ and $\partial_1 s_0+
s_{-1}\varepsilon_0 = \id_{P_0}$. As to the maps $s_1$ and $s_2$, we
don't need their full description to compute the cup product in the
cohomology ring. In fact, we don't need $s_2$ at all, as we shall see
now: having defined the maps $s_{-1}$ and $s_0$, we use
Propositions~\ref{proposition:tomodacontraction}
and~\ref{proposition:tomodadiagonal} to compute $\Delta_0$ and
$\Delta_1$. We get
\begin{align*}
\Delta_0 \colon &P_0 \to P_0\otimes P_0\\
&\Delta_0(x) = x\otimes x,\\
\Delta_1 \colon &P_1 \to (P_1\otimes P_0) \oplus (P_0\otimes P_1)\\
&\Delta_1(y_1) = y_1\otimes(a,1)x + x\otimes y_1,\\
&\Delta_1(y_2) = y_2\otimes(b,1)x + x\otimes y_2,\\
&\Delta_1(y_3) = y_3\otimes(1,t)x + x\otimes y_3.
\end{align*}
Let $\pi_{ij}\colon (P\otimes P)_{i+j} \to P_i\otimes P_j$ denote the
projection and $\Delta_{ij} = \pi_{ij}\circ\Delta_{i+j}: P_{i+j} \to
P_i\otimes P_j$. We observe that, for the computation of $H^1(G, \zz)
\otimes H^1(G, \zz) \stackrel{\smile}{\to} H^2(G, \zz\otimes\zz)$, we
need to know only the map $\Delta_{11} \colon P_2 \to P_1\otimes P_1$,
and the computation of $H^1(G, \zz) \otimes H^2(G, \zz)
\stackrel{\smile}{\to} H^3(G, \zz\otimes\zz)$ can be done once we have
$\Delta_{12} \colon P_3\to P_1\otimes P_2$.

From the resolution of Theorem~\ref{theorem:resolution} and
Propositions~\ref{proposition:tomodacontraction}
and~\ref{proposition:tomodadiagonal}, we can then verify that the maps
$\Delta_{11}$ and $\Delta_{12}$ can be calculated if we know how to
compute $s_1$ for the elements of the following list:
\begin{equation}\label{eq:listas1}
\begin{array}{l}
y_3, \\
(a,1)y_3, \\ 
(b,1)y_3, \\
(a^mb^n,1)y_1, \\
(a^mb^n,1)y_2, \\
(a^mb^n,t)y_1, \\
(a^mb^n,t)y_2. \\
\end{array}
\end{equation}

Before we compute $s_1$ for the elements of this list, we make one
more observation that will be useful later: if $M$ and $N$ are trivial
$\zz G$-modules, $g\in \zz G$, $m\in M$ and $f\in \Hom_{\zz G}(M,N)$,
then
\begin{equation}\label{qe:trivialaction}
f(gm) = g\cdot f(m) = \varepsilon(g)\cdot f(m),
\end{equation}
where $\varepsilon\colon \zz G\to \zz$ is the augmentation map.

\begin{lemma}\label{lemma:contraction}
Let
\[
\xymatrix{
0 \ar[r] & P_3 \ar[r]^-{\partial_3} & P_2 \ar[r]^-{\partial_2} & P_1 \ar[r]^-{\partial_1} & P_0 \ar[r]^-{\varepsilon_0} & {\zz} \ar[r] & 0
}
\]
be the free resolution of Theorem~\ref{theorem:resolution}. There is a
contracting homotopy $s$ for the resolution $P$ such that
\begin{align*}
s_{-1}(1) &= x,\\
s_0((a^mb^n,t^k)x) &= \frac{\partial (a^m,1)}{\partial (a,1)}y_1 + (a^m,1)\frac{\partial (b^n,1)}{\partial (b,1)}y_2 + (a^mb^n,1)\frac{\partial (1,t^k)}{\partial (1,y)}y_3,\\
s_1((a^mb^n,1)y_1) &= -(a^m,1)\frac{\partial(b^n,1)}{\partial(b,1)}z_3,\\
s_1((a^mb^n,1)y_2) &= 0, \\
s_1(y_3) &= 0, \\
s_1((a,1)y_3) &= 0, \\
s_1((b,1)y_3) &= 0, \\
s_1((a^mb^n,t)y_1) & = -(a^mb^n,1)\dfrac{\partial (a^\alpha,1)}{\partial (a,1)}z_1 - (a^{m+\alpha}b^n,1)\dfrac{\partial (b^\beta,1)}{\partial (b,1)}z_2 + \\
&\phantom{=}+ \left(-(a^m,1)\dfrac{\partial (a^\alpha,1)}{\partial (a,1)}\dfrac{\partial (b^n,1)}{\partial (b,1)} + \sum_{u,v\in\zz}h_{uv}(a^ub^v,t)\right)z_3, \\
s_1((a^mb^n,t)y_2) & = -(a^mb^n,1)\dfrac{\partial (a^\gamma,1)}{\partial (a,1)}z_1 - (a^{m+\gamma}b^n,1)\dfrac{\partial (b^\delta,1)}{\partial (b,1)}z_2 + \\
&\phantom{=}+ \left(-(a^m,1)\dfrac{\partial (a^\gamma,1)}{\partial (a,1)}\dfrac{\partial (b^n,1)}{\partial (b,1)} + \sum_{u,v\in\zz}q_{uv}(a^ub^v,t)\right)z_3,
\end{align*}
where the integers $h_{uv}$ e $q_{uv}$ satisfy
\begin{align*}
\sum_{u,v\in\zz} h_{uv} &= \frac{\alpha\beta}{2}\left(\gamma+\delta - \det\theta\right), \\
\sum_{u,v\in\zz} q_{uv} &= \frac{\gamma\delta}{2}\left(\alpha+\beta - \det\theta\right). \\
\end{align*}
\end{lemma}
\begin{proof}
We'we already defined the maps $s_{-1}$ and $s_0$. From
$\partial_2s_1+s_0\partial_1 = \id_{P_1}$, it is trivial to verify
that we can take $s_1(y_3) = s_1((a,1)y_3) = s_1((b,y)y_3) = 0$. It is
also easy to see that we can define $s_1((a^mb^n,1)y_1) =
-(a^m,1)\dfrac{\partial (b^n,1)}{\partial (b,1)}z_3$ and
$s_1((a^mb^n,1)y_2) = 0$.

Let us now write $s_1((a^mb^n,t)y_1) = k_1^{m,n}z_1 + k_2^{m,n}z_2 +
k_3^{m,n}z_3$, where $k_1^{m,n}, k_2^{m,n}, k_3^{m,n}\in \zz
G$. Substituting in $\partial_2(s_1((a^mb^m,t)y_1)) +
s_0(\partial_1((a^mb^n,t)y_1)) = (a^mb^n,t)y_1$, we get
\begin{equation}\label{eq:sistemas1ambnty1}
\left|
\begin{array}{l}
k_1^{m,n}A + k_2^{m,n}C + k_3^{m,n}[(1,1)-(b,1)] = (a^mb^n,t) - (a^m,1)\dfrac{\partial(a^\alpha,1)}{\partial(a,1)}, \\
k_1^{m,n}B + k_2^{m,n}D + k_3^{m,n}[(a,1)-(1,1)] = (a^m,1)\dfrac{\partial(b^n,1)}{\partial(b,1)} - (a^{m+\alpha},1)\dfrac{\partial(b^{n+\beta},1)}{\partial(b,1)}, \\
k_1^{m,n}[(a,1)-(1,1)] + k_2^{m,n}[(b,1)-(1,1)] = (a^mb^n,1) - (a^{m+\alpha}b^{n+\beta},1).
\end{array}
\right.
\end{equation}
The last equation above is satisfied for $k_1^{m,n} =
-(a^mb^n,1)\dfrac{\partial(a^\alpha,1)}{\partial(a,1)}$ and $k_2^{m,n}
= -(a^{m+\alpha}b^n,1)\dfrac{\partial(b^\beta,1)}{\partial(b,1)}$. We
must then find $k_3^{m,n}$ such that
\begin{align}\label{eq:k31}
k_3^{m,n}[(1,1)-(b,1)] &= (a^mb^n,t) - (a^m,1)\dfrac{\partial(a^\alpha,1)}{\partial(a,1)} - k_1^{m,n}A - k_2^{m,n}C  \notag\\
&= (a^mb^n,t) - (a^m,1)\dfrac{\partial(a^\alpha,1)}{\partial(a,1)} + (a^mb^n,1)\dfrac{\partial(a^\alpha,1)}{\partial(a,1)}A + (a^{m+\alpha}b^n,1)\dfrac{\partial(b^\beta,1)}{\partial(b,1)}C
\end{align}
and
\begin{align}\label{eq:k32}
k_3^{m,n}[(a,1)-(1,1)] &= (a^m,1)\dfrac{\partial(b^n,1)}{\partial (b,1)} - (a^{m+\alpha},1)\dfrac{\partial (b^{n+\beta},1)}{\partial (b,1)} - k_1^{m,n}B - k_2^{m,n}C  \notag\\
&= (a^m,1)\dfrac{\partial(b^n,1)}{\partial (b,1)} - (a^{m+\alpha},1)\dfrac{\partial (b^{n+\beta},1)}{\partial (b,1)} + (a^mb^n,1)\dfrac{\partial(a^\alpha,1)}{\partial (a,1)}B + \notag\\
&\phantom{=}\quad +(a^{m+\alpha}b^n,1)\dfrac{\partial(b^\beta,1)}{\partial (b,1)}D.
\end{align}
Now, if
\[
k_3^{m,n} = \sum_{u,v\in\zz} g_{uv}(a^ub^v,1) + \sum_{u,v\in\zz} h_{uv}(a^ub^v,t), \quad g_{uv},h_{uv}\in\zz,
\]
then the equations~(\ref{eq:k31})~and~(\ref{eq:k32}) are written as
\begin{align}
\sum_{u,v\in\zz} (g_{uv}-g_{u(v-1)})(a^ub^v,1) &= (a^mb^n,1)\dfrac{\partial(a^\alpha,1)}{\partial(a,1)} - (a^m,1)\dfrac{\partial(a^\alpha,1)}{\partial(a,1)}, \label{eq:guv1}\\
\sum_{u,v\in\zz} (g_{(u-1)v}-g_{uv})(a^ub^v,1) &= (a^m,1)\dfrac{\partial(b^n,1)}{\partial(b,1)} - (a^{m+\alpha},1)\dfrac{\partial(b^n,1)}{\partial(b,1)}, \label{eq:guv2}
\end{align}
and
\begin{align}
\sum_{u,v\in\zz}(h_{uv}-h_{(u-\gamma)(v-\delta)})(a^ub^v,t) &= (a^mb^n,t) + (a^mb^n,1)\dfrac{\partial(a^\alpha,1)}{\partial(a,1)}(A-(1,1)) + \notag\\
&\phantom{=} \quad +(a^{m+\alpha}b^n,1)\dfrac{\partial(b^\beta,1)}{\partial(b,1)}C, \label{eq:huv1}\\
\sum_{u,v\in\zz}(h_{(u-\alpha)(v-\beta)}-h_{uv})(a^ub^v,t) &= (a^mb^n,1)\dfrac{\partial(a^\alpha,1)}{\partial(a,1)}B + (a^{m+\alpha}b^n,1)\dfrac{\partial (b^\beta,1)}{\partial (b,1)}(D-(1,1)). \label{eq:huv2}
\end{align}
We can solve~(\ref{eq:guv1}) and~(\ref{eq:guv2}) in a way that is
similar to the one we used to find the element $E$ given by
equation~(\ref{eq:E}). We get
\begin{equation}\label{eq:guvresposta1}
\sum_{u,v\in\zz} g_{uv}(a^ub^v,1) = -(a^m,1)\dfrac{\partial(a^\alpha,1)}{\partial(a,1)}\dfrac{\partial(b^n,1)}{\partial(b,1)}.
\end{equation}
In order to solve~(\ref{eq:huv1}) and~(\ref{eq:huv2}), we first
observe that explicit expressions for the elements
\[
A,\:B,\:C,\:D, \:\frac{\partial(a^\alpha,1)}{\partial(a,1)}, \frac{\partial(b^\beta,1)}{\partial(b,1)}
\]
depend on the signs of $\alpha$, $\beta$, $\gamma$, $\delta$ and
$\det\theta$, so we have many cases to consider. As all of them are
similar, we show how to solve~(\ref{eq:huv1}) and~(\ref{eq:huv2}) when
$\alpha$, $\beta$, $\gamma$ and $\delta$ are positive and $\det\theta
= 1$. In this case we have
\begin{align*}
A-(1,1) &= - \sum_{k=0}^{\delta-1} (a^{k\alpha}b^{k\beta},t), \\
B &= \sum_{k=0}^{\beta-1} (a^{1+k\gamma}b^{k\delta},t), \\
C &= \sum_{k=1}^\gamma (a^{-k\alpha}b^{-k\beta},t), \\
D-(1,1) &= - \sum_{k=1}^\alpha (a^{-k\gamma}b^{1-k\delta},t),
\end{align*}
and the equations~(\ref{eq:huv1}) and~(\ref{eq:huv2}) can be written
as
\begin{align}\label{eq:huv1caso1}
\sum_{u,v\in\zz}(h_{uv}-h_{(u-\gamma)(v-\delta)})&(a^ub^v,t) = \notag\\
&= (a^mb^n,1)\left[(1,t) - \sum_{j=0}^{\alpha-1}\sum_{k=0}^{\delta-1}(a^{j+k\alpha}b^{k\beta},t) +
\sum_{j=0}^{\beta-1}\sum_{k=1}^\gamma(a^{\alpha-k\alpha}b^{j-k\beta},t)\right]
\end{align}
and
\begin{align}\label{eq:huv2caso1}
\sum_{u,v\in\zz}(h_{(u-\alpha)(v-\beta)}-h_{uv})&(a^ub^v,t) = \notag\\
&= (a^mb^n,1)\left[\sum_{j=0}^{\alpha-1}\sum_{k=0}^{\beta-1}(a^{1+j+k\gamma}b^{k\delta},t) - \sum_{j=0}^{\beta-1}\sum_{k=1}^\alpha (a^{\alpha-k\gamma}b^{1+j-k\delta},t)\right].
\end{align}
Just like we did to compute the value of the variable $E$ in the proof
of Theorem~\ref{theorem:resolution}, once more we define a directed
graph, with the set $\zz\times\zz$ as vertices and edges going from
each $(m,n)$ to $(m-\alpha,n-\beta)$ and $(m+\gamma,n+\delta)$. On the
edge from $(u-\gamma,v-\delta)$ to $(u,v)$, we put a label with the
value of $h_{uv}-h_{(u-\gamma)(v-\delta)}$, which is given by
equation~(\ref{eq:huv1caso1}) and on the edge from $(u,v)$ to
$(u-\alpha,v-\beta)$, we put a label with the value of
$h_{(u-\alpha)(v-\beta)}-h_{uv}$, which is given
by~(\ref{eq:huv2caso1}). In order to draw the graph, we apply again
the map $T\colon \zz\times\zz \to \zz\times\zz$ defined by
\[
\begin{bmatrix}
u\\v
\end{bmatrix} = T(x,y) =
\begin{bmatrix}
\alpha & \gamma\\
\beta & \delta
\end{bmatrix}^{-1}
\begin{bmatrix}
x\\y
\end{bmatrix},
\]
and we call the new coordinates, obtained after applying $T$, $u$ and
$v$.

In the case we are analyzing, let's take a look at the right
side of~(\ref{eq:huv2caso1}). As we are interested in calculating the
cup product with trivial coefficients $\zz$, $\zz_2$ and $\zz_p$, all
we need is the value of $\varepsilon(k_3^{m,n}) = \sum g_{uv} + \sum
h_{uv}$, where $\varepsilon \colon \zz G\to \zz$ is the augmentation
map. Since~(\ref{eq:guvresposta1}) already gives us $\sum g_{uv}$, all
we need is $\sum h_{uv}$, and for that we can drop the term
$(a^mb^n,1)$ in the right side of~(\ref{eq:huv2caso1}) and focus on
\[
\sum_{j=0}^{\alpha-1}\sum_{k=0}^{\beta-1}(a^{1+j+k\gamma}b^{k\delta},t) - \sum_{j=0}^{\beta-1}\sum_{k=1}^\alpha (a^{\alpha-k\gamma}b^{1+j-k\delta},t).
\]
The sum
$\sum_{j=0}^{\alpha-1}\sum_{k=0}^{\beta-1}(a^{1+j+k\gamma}b^{k\delta},t)$
provides us with edges with label $1$. For a fixed value of $j
\in\{0,\ldots,\alpha-1\}$, we get the following ``block'' of edges
(after applying $T$):

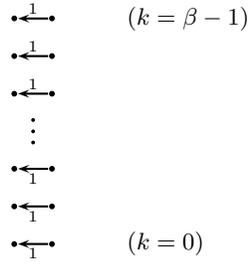
\begin{center}
\psset{unit=0.5cm}
\begin{pspicture}(0,0)(3,8)
\psdots[dotscale=0.6](1,1)(1,2)(1,3)(1,5)(1,6)(1,7)
\psdots[dotscale=0.6](2,1)(2,2)(2,3)(2,5)(2,6)(2,7)
\psline{->}(1.9,1)(1.1,1) \uput{0.06cm}[-90](1.5,1){$\scriptstyle 1$}
\psline{->}(1.9,2)(1.1,2) \uput{0.06cm}[-90](1.5,2){$\scriptstyle 1$}
\psline{->}(1.9,3)(1.1,3) \uput{0.06cm}[-90](1.5,3){$\scriptstyle 1$}
\psline{->}(1.9,5)(1.1,5) \uput{0.06cm}[90](1.5,5){$\scriptstyle 1$}
\psline{->}(1.9,6)(1.1,6) \uput{0.06cm}[90](1.5,6){$\scriptstyle 1$}
\psline{->}(1.9,7)(1.1,7) \uput{0.06cm}[90](1.5,7){$\scriptstyle 1$}
\psdots[dotscale=0.4](1.5,3.7)(1.5,4)(1.5,4.3)
\uput{1cm}[0](2,1){$(k=0)$}
\uput{1cm}[0](2,7){$(k=\beta-1)$}
\end{pspicture}
\end{center}
Now, as we let $j$ vary between $0$ and $\alpha-1$, other blocks of
edges with label $1$ appear, and are arranged according to the
following pattern, where each white rectangle represents a block of
edges for a fixed value of $j$:

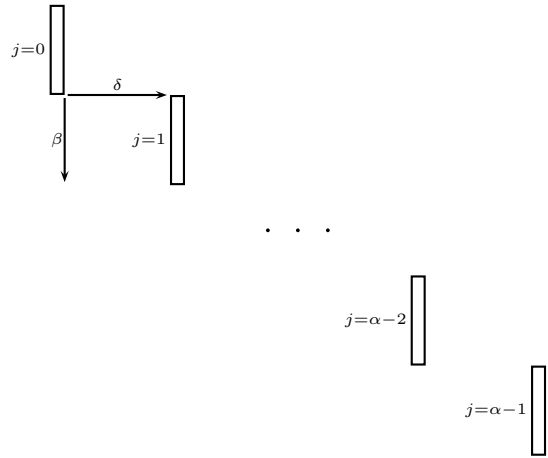
\begin{center}
\psset{unit=0.4cm}
\begin{pspicture}(0,0)(16.5,15)
\psframe(0,12)(0.5,15) \uput{0.06cm}[180](0,13.5){$\scriptstyle j=0$}
\psframe(4,9)(4.5,12) \uput{0.06cm}[180](4,10.5){$\scriptstyle j=1$}
\psdots[dotscale=0.4](8.25,7.5)(7.25,7.5)(9.25,7.5)
\psframe(12,3)(12.5,6) \uput{0.06cm}[180](12,4.5){$\scriptstyle j=\alpha-2$}
\psframe(16,0)(16.5,3) \uput{0.06cm}[180](16,1.5){$\scriptstyle j=\alpha-1$}
\psline{->}(0.6,12)(3.9,12) \uput{0.06cm}[90](2.25,12){$\scriptstyle \delta$}
\psline{->}(0.5,11.9)(0.5,9.1) \uput{0.06}[180](0.5,10.5){$\scriptstyle \beta$}
\end{pspicture}
\end{center}
If we do the same with the sum $-\sum_{j=0}^{\beta-1}\sum_{k=1}^\alpha
(a^{\alpha-k\gamma}b^{1+j-k\delta},t)$, which will give us edges
labeled with $-1$, we get a picture of the horizontal edges of our
graph (where the gray rectangles represent the edges of label $-1$):

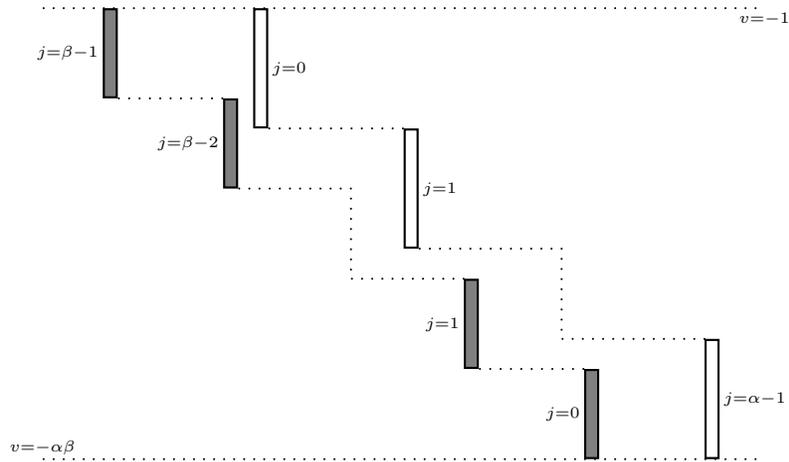
\begin{center}
\psset{unit=0.4cm}
\begin{pspicture}(-2,-1)(22.5,16)
\psframe[fillstyle=solid, fillcolor=gray](0,12)(0.5,15) \uput{0.06cm}[180](0,13.5){$\scriptstyle j=\beta-1$}
\psframe[fillstyle=solid, fillcolor=gray](4,9)(4.5,12) \uput{0.06cm}[180](4,10.5){$\scriptstyle j=\beta-2$}
\psframe[fillstyle=solid, fillcolor=gray](12,3)(12.5,6) \uput{0.06cm}[180](12,4.5){$\scriptstyle j=1$}
\psframe[fillstyle=solid, fillcolor=gray](16,0)(16.5,3) \uput{0.06cm}[180](16,1.5){$\scriptstyle j=0$}
\psframe(5,11)(5.5,15) \uput{0.06cm}[0](5.5,13){$\scriptstyle j=0$}
\psframe(10,7)(10.5,11) \uput{0.06cm}[0](10.5,9){$\scriptstyle j=1$}
\psframe(20,0)(20.5,4) \uput{0.06cm}[0](20.5,2){$\scriptstyle j=\alpha-1$}
\psset{linestyle=dotted}
\psline(5.5,11)(10,11)
\psline(10.5,7)(15.25,7)(15.25,4)(20,4)
\psline(0.5,12)(4,12)
\psline(4.5,9)(8.25,9)(8.25,6)(12,6)
\psline(12.5,3)(16,3)
\psline(-2,0)(22,0) \uput{0.06cm}[90](-2,0){$\scriptstyle v = -\alpha\beta$}
\psline(-2,15)(22,15) \uput{0.06cm}[-90](22,15){$\scriptstyle v = -1$}
\end{pspicture}
\end{center}

In order to label each vertex in a way that is consistent with the
labels on the edges, it is enough to put the label $1$ in every vertex
between a gray block and a white one, and put the label $0$ on every
other vertex. Since we can compute the exact location of the blocks in
our graph, we can then write
\begin{align*}
\sum_{u,v\in\zz} h_{uv} &=
\left(
\begin{array}{p{4cm}}
sum of the $u$-coordinates of the origins of vertices with label $1$
\end{array}
\right) -
\left(
\begin{array}{p{4cm}}
sum of the $u$-coordinates of the origins of vertices with label $-1$
\end{array}
\right) \\
&= \frac{(\alpha\delta+\delta)\alpha\beta}{2} -
\frac{(\alpha\delta-\gamma+1)\alpha\beta}{2} = \frac{\alpha\beta}{2}(\gamma+\delta-1).
\end{align*}
One thing is missing: we need to check that our labels on the vertices
are also consistent with equation~(\ref{eq:huv1caso1}), but
fortunately they are.

If we now analyze all the remaining cases, we get
\[
\sum_{u,v\in\zz} h_{uv} =
\begin{cases}
\dfrac{\alpha\beta}{2}(\gamma+\delta-1), & \text{if $\det\theta=1$,}\\
\dfrac{\alpha\beta}{2}(\gamma+\delta+1), & \text{if $\det\theta=-1$.}
\end{cases}
\]
Also, the computation of $s_1((a^mb^n,t)y_2)$ is analogous to that of
$s_1((a^mb^n,t)y_1)$, and what we get is exactly what is in the
statement of this lemma.
\end{proof}

\begin{remark}
If we are interested in computing the cup product with non trivial
coefficients, then we can still construct the graphs described in the
lemma for a given action $\theta$ and explicitly calculate
$s_1((a^mb^n,t)y_1)$ and $s_1((a^mb^n,t)y_2)$.
\end{remark}

Now we proceed to calculate the cup product $H^p(G,\zz) \otimes
H^q(G,\zz) \stackrel{\smile}{\to} H^{p+q}(G,\zz)$. Just to simplify
the notation, let's write the elements $A$, $B$, $C$, $D$ and $E\in
\zz G$ as
\[
A = \sum \pm A_k, \quad B = \sum \pm B_k, \quad C = \sum \pm C_k,
\quad D = \sum \pm D_k, \quad E = \sum \pm E_k.
\]
Then, using Propositions~\ref{proposition:tomodacontraction}
and~\ref{proposition:tomodadiagonal}, we can write
\begin{align}\label{eq:delta2}
&\Delta_2(z_1) = \tilde{s}_1\Delta_1\varphi_2(z_1) = \tilde{s}_1\bigl(A\Delta_1(y_1) + B\Delta_1(y_2) + [(a,1)-(1,1)]\Delta_1(y_3)\bigr) \notag\\
&\phantom{\Delta_2(z_1)} = \tilde{s}_1\bigl(A[y_1\otimes(a,1)x + x\otimes y_1] + B[y_2\otimes(b,1)x + x\otimes y_2] + \notag\\
&\phantom{\Delta_2(z_1) = \tilde{s}_1\bigl(} +(a,1)y_3\otimes(a,t)x + (a,1)x\otimes(a,1)y_3 - y_3\otimes(1,t)x - x\otimes y_3 \bigr), \notag\\
&\Delta_2(z_2) = \tilde{s}_1\Delta_1\varphi_2(z_2) = \tilde{s}_1\bigl(C\Delta_1(y_1) + D\Delta_1(y_2) + [(b,1)-(1,1)]\Delta_1(y_3)\bigr) \notag\\
&\phantom{\Delta_2(z_2)} = \tilde{s}_1\bigl(C[y_1\otimes(a,1)x + x\otimes y_1] + D[y_2\otimes(b,1)x + x\otimes y_2] + \\
&\phantom{\Delta_2(z_2) = \tilde{s}_1\bigl(} +(b,1)y_3\otimes(b,t)x + (b,1)x\otimes(b,1)y_3 - y_3\otimes(1,t)x - x\otimes y_3 \bigr), \notag\\
&\Delta_2(z_3) = \tilde{s}_1\Delta_1\varphi_2(z_3) = \tilde{s}_1\bigl(\Delta_1(y_1) - (b,1)\Delta_1(y_1) + (a,1)\Delta_1(y_2) - \Delta_1(y_2) \bigr) \notag\\
&\phantom{\Delta_2(z_3)} = s_1(y_1)\otimes(a,1)x + x\otimes s_1(y_1) - s_1((b,1)y_1)\otimes(ab,1)x - s_0((b,1)x)\otimes(b,1)y_1 \notag\\
&\phantom{\Delta_2(z_3) =} - x\otimes s_1((b,1)y_1) + s_1((a,1)y_2)\otimes(ab,1)x + s_0((a,1)x)\otimes(a,1)y_2 \notag\\
&\phantom{\Delta_2(z_3) =} +x\otimes s_1((a,1)y_2) - s_1(y_2)\otimes(b,1)x - x\otimes s_1(y_2), \notag
\end{align}
and we get
\begin{align}\label{eq:delta11}
\Delta_{11}(z_1) &= \pi_{11}(\Delta_2(z_1)) = \pi_{11}\circ\tilde{s}_1(A(x\otimes y_1) + B(x\otimes y_2) + (a,1)x\otimes(a,1)y_3 - x\otimes y_3) \notag\\
&= \pm\sum s_0(A_kx)\otimes A_ky_1 \pm \sum s_0(B_kx)\otimes B_ky_2 + y_1\otimes (a,1)y_3, \notag\\
\Delta_{11}(z_2) &= \pi_{11}(\Delta_2(z_2)) = \pi_{11}\circ\tilde{s}_1(C(x\otimes y_1) + D(x\otimes y_2) + (b,1)x\otimes(b,1)y_3 - x\otimes y_3) \\
&= \pm\sum s_0(C_kx)\otimes C_ky_1 \pm \sum s_0(D_kx)\otimes D_ky_2 + y_2\otimes (b,1)y_3, \notag\\
\Delta_{11}(z_3) &= \pi_{11}(\Delta_2(z_3)) = -s_0((b,1)x)\otimes(b,1)y_1 + s_0((a,1)x)\otimes(a,1)y_2 \notag \\
&= y_1\otimes (a,1)y_2 - y_2\otimes (b,1)y_1. \notag
\end{align}

Now, if $u,v\in \Hom_{\zz G}(P_1,\zz)$, equation~(\ref{eq:delta11})
implies that the product $[u]\smile [v]\in H^2(G,\zz)$ is represented
by a map $(u\smile v) \in \Hom_{\zz G}(P_2,\zz)$ such that
\begin{align}\label{eq:h1cuph1}
(u\smile v)(z_1) &= (u\times v)\Delta_{11}(z_1) =\notag\\
&= \left(-\dfrac{\alpha m_1(m_1+1)}{2}\right)u(y_1)\otimes v(y_1) +
\left(-\dfrac{\beta m_1(m_1+1)}{2}\right)u(y_2)\otimes v(y_1) +
m_1u(y_3)\otimes v(y_1) + \notag\\
&\phantom{=}+ \left(n_1+\dfrac{\gamma n_1(n_1-1)}{2}\right)u(y_1)\otimes v(y_2) +
\left(\dfrac{\delta n_1(n_1-1)}{2}\right)u(y_2)\otimes v(y_2) + n_1u(y_3)\otimes v(y_2) +\notag\\
&\phantom{=}+ u(y_1)\otimes v(y_3), \notag\\
(u\smile v)(z_2) &= (u \times v)\Delta_{11}(z_2) =\\
&= \left(-\dfrac{\alpha m_2(m_2+1)}{2}\right)u(y_1)\otimes v(y_1) +
\left(-\dfrac{\beta m_2(m_2+1)}{2}\right)u(y_2)\otimes v(y_1) +
m_2u(y_3)\otimes v(y_1) + \notag\\
&\phantom{=}+ \left(\dfrac{\gamma n_2(n_2-1)}{2}\right)u(y_1)\otimes v(y_2) +
\left(n_2 + \dfrac{\delta n_2(n_2-1)}{2}\right)u(y_2)\otimes v(y_2) +
n_2u(y_3)\otimes v(y_2) +\notag\\
&\phantom{=}+ u(y_2)\otimes v(y_3), \notag\\
(u\smile v)(z_3) &= u(y_1)\otimes v(y_2) - u(y_2)\otimes v(y_1). \notag
\end{align}

Even though the expressions for $A$, $B$, $C$, and $D$ depend on the
signs of $m_1$, $n_1$, $m_2$ and $n_2$, the above equations hold in
every case.

Using once again Propositions~\ref{proposition:tomodacontraction}
and~\ref{proposition:tomodadiagonal}, we also write
\begin{align}\label{eq:delta3}
\Delta_3(w) &= \tilde{s}_2\Delta_2\varphi_3(w) = \tilde{s}_2\Delta_2\bigl([(1,1)-(b,1)]z_1 + [(a,1)-(1,1)]z_2 + Ez_3\bigr) \notag\\
&= \tilde{s}_2\bigl([(1,1)-(b,1)]\Delta_2(z_1) + [(a,1)-(1,1)]\Delta_2(z_2) + E\Delta_2(z_3)\bigr).
\end{align}

A term belonging to $P_1\otimes P_2$ in $\Delta_3(w)$ arises when we
evaluate $\tilde{s}_2$ at an element of $P_0\otimes P_2$ (this follows
from Proposition~\ref{proposition:tomodacontraction}). Hence we have
$\Delta_{12}(w) = \pi_{12}\circ\Delta_3(w) = \pi_{12}\circ
\tilde{s}_2\circ \Delta_{02}\circ \partial_3(w)$, and looking
at~(\ref{eq:delta2}), we get
\begin{align}\label{eq:delta12expandido}
\Delta_{12}(w) &= \pi_{12}\circ \tilde{s}_2(\Delta_{02}(z_1)-(b,1)\Delta_{02}(z_1) + (a,1)\Delta_{02}(z_2) - \Delta_{02}(z_2) + E\Delta_{02}(z_3)) \notag\\
&= \pm \sum y_1\otimes (a,1)s_1(C_ky_1) \pm \sum y_1\otimes (a,1)s_1(D_ky_2) \\
&\phantom{ = }\mp \sum y_2\otimes (b,1)s_1(A_ky_1) \mp \sum y_2\otimes (b,1)s_1(B_ky_2) \notag\\
&\phantom{ = }+ \sum \pm s_0(E_kx)\otimes E_kz_3. \notag
\end{align}

If $u\in \Hom_{\zz G}(P_1,\zz)$ and $v\in \Hom_{\zz G}(P_2,\zz)$, then
$[u]\smile [v]\in H^3(G,\zz)$ is represented by a map $u\smile v \in
\Hom_{\zz G}(P_3,\zz)$ such that $(u\smile v)(w) = (u\times
v)\Delta_{12}(w)$. Using the above expression for $\Delta_{12}(w)$, we
can then write
\begin{align*}
(u\times v)\left(\sum\pm s_0(E_kx)\otimes E_kz_3\right) &= \sum_{(m,n)\in I_1\times J_1}\!\!\!\!\!\!\!\!\!(\alpha m+\gamma n)u(y_1)\otimes v(z_3)+
\sum_{(m,n)\in I_1\times J_1}\!\!\!\!\!\!\!\!\! (\beta m +\delta n)u(y_2)\otimes v(z_3) \\
&\phantom{=} + \sum_{(m,n)\in I_1\times J_1}\!\!\!\!\!\!\!\!\! u(y_3)\otimes v(z_3) \\
&\phantom{=} - \sum_{(m,n)\in I_2\times J_2}\!\!\!\!\!\!\!\!\! (\alpha m + \gamma n)u(y_1)\otimes v(z_3)
- \sum_{(m,n)\in I_2\times J_2}\!\!\!\!\!\!\!\!\! (\beta m + \delta n)u(y_2)\otimes v(z_3) \\
&\phantom{=} - \sum_{(m,n)\in I_2\times J_2}\!\!\!\!\!\!\!\!\! u(y_3)\otimes v(z_3) \\
&= \left(\alpha|J_1|\sum_{m\in I_1}\!\!m + \gamma|I_1|\sum_{n\in J_1}\!\!n - \alpha|J_2|\sum_{m\in I_2}\!\!m - \gamma|I_2|\sum_{n\in J_2}\!\!n\right)u(y_1)\otimes v(z_3) \\
&\phantom{=}+ \left(\beta|J_1|\!\sum_{m\in I_1}\!\!m + \delta|I_1|\!\sum_{n\in J_1}\!\!n - \beta|J_2|\!\sum_{m\in I_2}\!\!m - \delta|I_2|\! \sum_{n\in J_2}\!\!n\right)u(y_2)\otimes v(z_3) \\
&\phantom{=}+ (\det\theta) u(y_3)\otimes v(z_3),
\end{align*}
and similar calculations show that
\begin{align*}
(u\times v)\left(\pm \sum y_1\otimes (a,1)s_1(C_ky_1)\right) &= (-\alpha m_2)u(y_1)\otimes v(z_1) + (-\beta m_2)u(y_1)\otimes v(z_2) \\
&\phantom{=} + \left(\frac{\alpha\beta m_2}{2}\left(\gamma + \delta + m_2 + 1 - \det\theta\right)\right) u(y_1)\otimes v(z_3), \\
(u\times v)\left(\pm \sum y_1\otimes (a,1)s_1(D_ky_2)\right) &= (-\gamma n_2)u(y_1)\otimes v(z_1) + (-\delta n_2)u(y_1)\otimes v(z_2) \\
&\phantom{=} + \left(-\gamma n_2 - \frac{\gamma\delta n_2}{2}\left(-\alpha - \beta + n_2 - 1 + \det\theta\right)\right) u(y_1)\otimes v(z_3),\\
(u\times v)\left(\mp \sum y_2\otimes (b,1)s_1(A_ky_1)\right) &= (\alpha m_1)u(y_2)\otimes v(z_1) + (\beta m_1)u(y_2)\otimes v(z_2) \\
&\phantom{=} + \left(\frac{-\alpha\beta m_1}{2}\left(\gamma + \delta + m_1 + 1 - \det\theta\right)\right) u(y_2)\otimes v(z_3), \\
(u\times v)\left(\mp \sum y_2\otimes (b,1)s_1(B_ky_2)\right) &= (\gamma n_1)u(y_2)\otimes v(z_1) + (\delta n_1)u(y_2)\otimes v(z_2) \\
&\phantom{=} + \left(\frac{\gamma\delta n_1}{2}\left(-\alpha - \beta + n_1 - 1 + \det\theta\right)\right) u(y_2)\otimes v(z_3).
\end{align*}

Substituting those results in~(\ref{eq:delta12expandido}), we are left
with
\begin{align}\label{eq:h1cuph2}
(u\smile v)(w) &= (u\times v)\Delta_{12}(w) \notag\\
&= u(y_1)\otimes v(z_2) - u(y_2)\otimes v(z_1) +\notag\\
&\phantom{=}+ \biggl[\frac{\alpha\beta m_2}{2}\left(\gamma + \delta + m_2 + 1 - \det\theta\right) -\gamma n_2 - \frac{\gamma\delta n_2}{2}\left(-\alpha - \beta + n_2 - 1 + \det\theta\right) +\notag\\
&\phantom{=+}\quad \alpha|J_1|\sum_{m\in I_1}\!\!m + \gamma|I_1|\sum_{n\in J_1}\!\!n - \alpha|J_2|\sum_{m\in I_2}\!\!m - \gamma|I_2|\sum_{n\in J_2}\!\!n \biggr]u(y_1)\otimes v(z_3)+\notag\\
&\phantom{=}+\biggl[\frac{-\alpha\beta m_1}{2}\left(\gamma + \delta + m_1 + 1 - \det\theta\right) + \frac{\gamma\delta n_1}{2}\left(-\alpha - \beta + n_1 - 1 + \det\theta\right) +\notag\\
&\phantom{=+}\quad \beta|J_1|\sum_{m\in I_1}\!\!m + \delta|I_1|\sum_{n\in J_1}\!\!n - \beta|J_2|\sum_{m\in I_2}\!\!m - \delta|I_2|\sum_{n\in J_2}\!\!n \biggr]u(y_2)\otimes v(z_3) +\notag\\
&\phantom{=} + (\det\theta)u(y_3)\otimes v(z_3).
\end{align}

The above equation can be simplified: first, we note that
\begin{gather*}
\frac{\alpha\beta m_2}{2}\left(\gamma + \delta + m_2 + 1 - \det\theta\right) -\gamma n_2 - \frac{\gamma\delta n_2}{2}\left(-\alpha - \beta + n_2 - 1 + \det\theta\right) =\\
=\begin{cases}
0, & \text{ if $\det\theta = 1$}\\
\alpha\gamma(\delta-\beta-1), & \text{ if $\det\theta = -1$},
\end{cases}
\end{gather*}
and
\begin{gather*}
\frac{-\alpha\beta m_1}{2}\left(\gamma + \delta + m_1 + 1 - \det\theta\right) + \frac{\gamma\delta n_1}{2}\left(-\alpha - \beta + n_1 - 1 + \det\theta\right) = \\
=\begin{cases}
0, & \text{ if $\det\theta = 1$},\\
\beta\delta(\gamma-\alpha+1), & \text{ if $\det\theta = -1$}.
\end{cases}
\end{gather*}
We also have the following:

\begin{lemma}\label{lema:Sxy}
Let $S\colon \zz\times\zz \to \zz$ be the function defined by
\[
S(x,y) = x|J_1|\sum_{m\in I_1}m + y|I_1|\sum_{n\in J_1}n - x|J_2|\sum_{m\in I_2}m - y|I_2|\sum_{n\in J_2}n,
\]
where the sets $I_1$, $J_1$, $I_2$ and $J_2$ are described in
Table~\ref{table:ij}. Then
\[
S(\alpha,\gamma) =
\begin{cases}
\dfrac{1-\alpha-\gamma-\alpha\gamma}{2}, & \text{{\vrule height 16pt depth 16pt width 0pt} if $\det\theta=1$},\\
\dfrac{-1+\alpha+\gamma-\alpha\gamma}{2}, & \text{{\vrule height 16pt depth 16pt width 0pt} if $\det\theta=-1$},
\end{cases}
\]
and
\[
S(\beta,\delta) =
\begin{cases}
\dfrac{1-\beta-\delta+\beta\delta}{2}, & \text{{\vrule height 16pt depth 16pt width 0pt} if $\det\theta=1$}, \\
\dfrac{-1+\beta+\delta+\beta\delta}{2}, & \text{{\vrule height 16pt depth 16pt width 0pt} if $\det\theta=-1$}. \\
\end{cases}
\]
\end{lemma}
\begin{proof}
The proof is nothing but a straightforward calculation using the
descriptions of the sets $I_1$, $J_1$, $I_2$ and $J_2$ and remembering that
\[
\begin{bmatrix}
m_1 & m_2 \\
n_1 & n_2
\end{bmatrix} = 
-\begin{bmatrix}
\alpha & \gamma \\
\beta & \delta
\end{bmatrix}^{-1} =
\frac{1}{\alpha\delta-\beta\gamma}
\begin{bmatrix}
-\delta & \gamma \\
\beta & -\alpha
\end{bmatrix}.
\]
\end{proof}

With those two observations, equation~(\ref{eq:h1cuph2}) can then be
written, if $\det\theta=1$, as
\begin{align}\label{eq:h1cuph2det1}
(u\smile v)(w) &= u(y_1)\otimes v(z_2) - u(y_2)\otimes v(z_1) + \left(\frac{1-\alpha-\gamma-\alpha\gamma}{2}\right)u(y_1)\otimes v(z_3) +\notag\\
&\phantom{=+}\left(\frac{1-\beta-\delta+\beta\delta}{2}\right)u(y_2)\otimes v(z_3) + u(y_3)\otimes v(z_3),
\end{align}
and, if $\det\theta=-1$, we can write
\begin{align}\label{eq:h1cuph2det-1}
(u\smile v)(w) &= u(y_1)\otimes v(z_2) - u(y_2)\otimes v(z_1) + \notag\\
&\phantom{=+}\biggl[\alpha\gamma(\delta-\beta-1) + \left(\frac{-1+\alpha+\gamma-\alpha\gamma}{2}\right)\biggr]u(y_1)\otimes v(z_3) +\\
&\phantom{=+}\biggl[\beta\delta(\gamma-\alpha+1) + \left(\frac{-1+\beta+\delta+\beta\delta}{2}\right)\biggr]u(y_2)\otimes v(z_3) - u(y_3)\otimes v(z_3). \notag
\end{align}

We are now finally in a position where we can compute the products
\[
H^p(G,\zz) \otimes H^q(G,\zz) \stackrel{\smile}{\to} H^{p+q}(G,\zz\otimes\zz) \cong H^{p+q}(G,\zz).
\]

\begin{theorem}\label{theorem:cupZ}
The cohomology ring $H^*(G,\zz)$ is given by:
\begin{enumerate}
\item If $\rank(\theta-I)=0$, then
\[
H^*(G,\zz) \cong \dfrac{\zz[\zeta_1,\zeta_2,\zeta_3]}{(\zeta_1^2 = \zeta_2^2 = \zeta_3^2 = 0)},
\]
where $\dim(\zeta_1) = \dim(\zeta_2) = \dim(\zeta_3) = 1$.

\item If $\rank(\theta-I)=1$ and $\det\theta=1$, then
\[
H^*(G,\zz) \cong \dfrac{\zz[\zeta_1,\zeta_2,\xi_2, \xi_3]}
{\begin{pmatrix}\zeta_1^2
  = \zeta_2^2 = 0,\: \gcd(\beta,\gamma)\zeta_1\zeta_2 = 0,\:
  \zeta_2\xi_2 = 0, \: \zeta_2\xi_3=\zeta_1\xi_2, \\
  \zeta_1\xi_3 = \frac{1}{2}\sqrt{\frac{|\beta|}{\gcd(\beta,\gamma)}}(1-\alpha-\alpha\gamma-\beta\gamma)\zeta_1\xi_2, \\
  \xi_2^2 = \xi_3^2 = \xi_2\xi_3 = 0
\end{pmatrix}},
\]
where $\dim(\zeta_1) = \dim(\zeta_2) = 1$ and $\dim(\xi_2) =
\dim(\xi_3) = 2$.

\item If $\rank(\theta-I)=1$, $\det\theta=-1$ and $\gcd(\beta,\gamma,2)=1$, then
\[
H^*(G,\zz) \cong \dfrac{\zz[\zeta_1,\zeta_2,\xi]}{\begin{pmatrix}
    \zeta_1^2 = \zeta_2^2 = 0,\: \zeta_1\zeta_2 = 2\xi,\\
    \zeta_2\xi=0,\: 2\zeta_1\xi=0,\: \xi^2=0
  \end{pmatrix}},
\]
where $\dim(\zeta_1) = \dim(\zeta_2) = 1$ and $\dim(\xi)=2$.

\item If $\rank(\theta-I)=1$, $\det\theta=-1$ and $\gcd(\beta,\gamma,2)=2$, then
\[
H^*(G,\zz) \cong \dfrac{\zz[\zeta_1,\zeta_2,\xi_1,\xi_2]}{\begin{pmatrix}
      \zeta_1^2=\zeta_2^2=0,\: \zeta_1\zeta_2 = m\xi_1+\xi_2, 2\xi_1=0 \\
      \zeta_1\xi_2 = \left(\frac{kr'-\ell s'}{2}\right)\zeta_1\xi_1, \zeta_2\xi_1=0, \\
      \zeta_2\xi_2=0, \xi_1^2=\xi_2^2=\xi_1\xi_2=0
    \end{pmatrix}},
\]
where $\dim(\zeta_1) = \dim(\zeta_2) = 1$, $\dim(\xi_1) = \dim(\xi_2) =
2$.
Also, the integers $k$ and $\ell$ are such that $pk+q\ell=1$, where
$p$ and $q$ are relatively prime integers satisfying $\dfrac{p}{q} =
\dfrac{\gamma}{1+\delta} = \dfrac{1+\alpha}{\beta}$. Finally, the
integers $r'$, $s'$ and $m$ are such that $ps'+qr'=2$,
$m=\dfrac{2k-s'}{2q} = \dfrac{r'-2\ell}{2p}$.

\item If $\rank(\theta-I)=2$ e $\det\theta=1$, then
\[
H^*(G,\zz) \cong \dfrac{\zz[\zeta,\xi_1,\xi_2,\xi_3]}{\begin{pmatrix}
    \zeta^2=0,\: c_1\xi_1=c_2\xi_2=0, \:\zeta\xi_1=\zeta\xi_2=0,\\
    \xi_1^2=\xi_2^2=\xi_3^2=0, \xi_1\xi_2=\xi_1\xi_3=\xi_2\xi_3=0
\end{pmatrix}},
\]
where $\dim(\zeta)=1$ and $\dim(\xi_1) = \dim(\xi_2) = \dim(\xi_3) =
2$. Also, the integers $c_1$ and $c_2$ are such that $c_1\mid c_2$ and
$c_1c_2 = |\det(\theta-I)|$.
\item If $\rank(\theta-I)=2$ and $\det\theta=-1$, then
\[
H^*(G,\zz) \cong \dfrac{\zz[\zeta,\xi_1,\xi_2,\chi]}{\begin{pmatrix}
    \zeta^2=0,\: c_1\xi_1=c_2\xi_2=0,\: \zeta\xi_1=0,\: \zeta\xi_2=0,\\
    \zeta\chi=0,\: \xi_1^2=\xi_2^2=0,\: \xi_1\chi=0,\: \xi_2\chi=0,\: \chi^2=0
\end{pmatrix}},
\]
where $\dim(\zeta) = 1$, $\dim(\xi_1)=\dim(\xi_2)=2$, $\dim(\chi)=3$.
Also, the integers $c_1$ and $c_2$ are such that $c_1\mid c_2$ and
$c_1c_2 = |\det(\theta-I)|$.
\end{enumerate}
\end{theorem}

\begin{proof}
\noindent{\it First case:\/ } if $\rank(\theta-I) = 0$, then
\begin{align*}
H^1(G,\zz) &= \gen{[y_1^*]}\oplus\gen{[y_2^*]}\oplus\gen{[y_3^*]} \cong \zz\oplus\zz\oplus\zz,\\
H^2(G,\zz) &= \gen{[z_1^*]}\oplus\gen{[z_2^*]}\oplus\gen{[z_3^*]} \cong \zz\oplus\zz\oplus\zz,\\
H^3(G,\zz) &= \gen{[w^*]} \cong \zz.
\end{align*}
and equations~(\ref{eq:h1cuph1}) and~(\ref{eq:h1cuph2det1}) give us
the cohomology ring $H^*(G,\zz)$ exactly like in the statement of the
theorem, with $[y_j^*] = \zeta_j$.

\bigskip
\noindent{\it Second case:\/} $\rank(\theta-I)=1$ and $\det\theta =
1$. Just like in the proof of Theorem~\ref{theorem:HZ}, suppose $1+m_1
\ne 0$ oor $n_1\ne 0$ (the case where $1+m_1=n_1=0$ is analogous). We have
\[
I-\theta^{-1} = \begin{bmatrix}
1+m_1 & m_2 \\
n_1 & 1+n_2
\end{bmatrix} = 
\begin{bmatrix}
qr' & pr' \\
qs' & ps'
\end{bmatrix},
\]
where the integers $p$ and $q$ satisfy $\gcd(p,q)=1$ and $p,q\ne 0$. The
generators of $H^1(G,\zz) \cong \zz\oplus\zz$ are $[u]$ and
$[y_3^*]$, where
\[
u = -\frac{qs'}{\gcd(qr',qs')}y_1^* + \frac{qr'}{\gcd(qr',qs')}y_2^* = -\frac{s'}{\gcd(r',s')}y_1^* + \frac{r'}{\gcd(r',s')}y_2^*.
\]
We can say even more: we have $\det(I-\theta^{-1}) = 0 \iff
\alpha+\delta = 2$. But $\alpha+\delta=2 \iff (1-qr')+(1-ps') = 2 \iff
qr'+ps' = 0$. Hence $p\mid r'$ and $q\mid s'$. Writing $r' = pr''$ and
$s'=qs''$, we then get $pqr''+pqs'' = 0 \iff s''=-r''$. We can then
write
\begin{equation}\label{eq:r''}
I-\theta^{-1} = \begin{bmatrix}
1+m_1 & m_2 \\
n_1 & 1+n_2
\end{bmatrix} =
\begin{bmatrix}
(1-\delta) & \gamma \\
\beta & (1-\alpha)
\end{bmatrix} =
\begin{bmatrix}
pqr'' & p^2r''\\
-q^2r'' & -pqr''
\end{bmatrix}
\end{equation}
and also $\gcd(\beta,\gamma) = \gcd(r',s') = |r''|$, so
\[
u = -\frac{s'}{\gcd(r',s')}y_1^* + \frac{r'}{\gcd(r',s')}y_2^* = qy_1^*+py_2^*.
\]
Using equation~(\ref{eq:h1cuph1}), we obtain
\begin{align*}
(u\smile u)(z_1) &= -\frac{(1+pqr'')(pqr'')(pqr''-1)q^2}{2} + \frac{(q^2r'')(pqr'')(pqr''-1)pq}{2} + \\
&\phantom{=}+\left(-q^2r'' + \frac{(p^2r'')(-q^2r'')(-q^2r''-1)}{2}\right)pq + \frac{(1-pqr'')(-q^2r'')(-q^2r''-1)p^2}{2} \\
&= \frac{pq^2(p-q)r''}{2}, \\
(u\smile u)(z_2) &= \frac{p^2q(p-q)r''}{2}, \\
(u\smile u)(z_3) &= 0.
\end{align*}
Thus $(u\smile u) = \dfrac{pq(p-q)r''}{2}(qz_1^*+pz_2^*)$. We know
that $H^2(G,\zz) \cong \zz_{\gcd(\beta,\gamma)}\oplus \zz\oplus \zz
\cong \zz_{r''}\oplus\zz\oplus\zz$. The generator of one of the $\zz$
factors in $H^2(G,\zz)$ is $[z_3^*]$. The other two generators are
obtained from the Smith normal form of
\[
I-\theta^{-1} = \begin{bmatrix}
1+m_1 & m_2 \\
n_1 & 1+n_2
\end{bmatrix} =
\begin{bmatrix}
pqr'' & p^2r''\\
-q^2r'' & -pqr''
\end{bmatrix}.
\]
Let $k$ and $\ell$ be integers such that $pk+q\ell=1$. We have
$\begin{bmatrix}\ell & -p \\ k & q\end{bmatrix} \in GL_2(\zz)$ and
\[
\begin{bmatrix}\ell & -p \\ k & q\end{bmatrix}^{-1} =
\begin{bmatrix}q & p \\ -k & \ell\end{bmatrix},
\]
which gives us
\[
(I-\theta^{-1})\begin{bmatrix}\ell & -p \\ k & q\end{bmatrix} =
\begin{bmatrix}pr'' & 0 \\ -qr'' & 0\end{bmatrix} \quad\text{and}\quad
\begin{bmatrix}\ell & -p \\ k & q\end{bmatrix}^{-1}
\begin{bmatrix}z_1^*\\z_2^*\end{bmatrix} =
\begin{bmatrix}qz_1^* + pz_2^*\\ -kz_1^*+\ell z_2^*\end{bmatrix}.
\]
Hence the generator of the factor $\zz_{r''}$ of $H^2(G,\zz)$ is the
class of the map $qz_1^*+pz_2^*$, and the generator of the other
factor $\zz$ of $H^2(G,\zz)$ is the class of $-kz_1^*+\ell
z_2^*$. As $pq(p-q)$ is even, we get
\[
(u\smile u) = \left(\dfrac{pq(p-q)}{2}\right)\cdot r''(qz_1^*+pz_2^*),
\]
which means that $[u]^2=0$. We also have $[y_3^*]^2=0$. Let us compute
$[u]\smile[y_3^*]$: using~(\ref{eq:h1cuph1}), we obtain
\begin{align*}
(u\smile y_3^*)(z_1) &= q,\\
(u\smile y_3^*)(z_2) &= p,\\
(u\smile y_3^*)(z_3) &= 0,
\end{align*}
so $[u]\smile[y_3^*] = [qz_1^*+pz_2^*]$.

Now, to compute the products $H^1(G,\zz)\otimes H^2(G,\zz)
\stackrel{\smile}{\to} H^3(G,\zz)$, we use~(\ref{eq:h1cuph2det1}),
that gives us
\[
\begin{array}{l}
{[u]} \smile [qz_1^*+pz_2^*] = 0, \\
{[u]} \smile [-kz_1^*+\ell z_2^*] = [w^*], \\
{[y_3^*]} \smile [qz_1^*+pz_2^*] = 0, \\
{[y_3^*]} \smile [-kz_1^*+\ell z_2^*] = 0, \\
{[y_3^*]} \smile [z_3^*] = [w^*].
\end{array}
\]
We are then left with the computation of ${[u]} \smile [z_3^*]$. Equation~(\ref{eq:h1cuph2det1}) gives us
\begin{equation}\label{eq:uz3}
(u\smile z_3^*)(w) = \left(\frac{1-\alpha-\gamma-\alpha\gamma}{2}\right)q+\left(\frac{1-\beta-\delta+\beta\delta}{2}\right)p.
\end{equation}

Assuming that $q\ge 0$, it follows from~(\ref{eq:r''}) that
\[
r'' = -\dfrac{\beta}{|\beta|}\gcd(\beta,\gamma) = \dfrac{\gamma}{|\gamma|}\gcd(\beta,\gamma)
\quad\text{and}\quad
q = \sqrt{\frac{-\beta}{r''}} = \sqrt{\frac{|\beta|}{\gcd(\beta,\gamma)}}.
\]
We also have
\[
\dfrac{p}{q} = \dfrac{1-\alpha}{\beta} \iff p = \sqrt{\frac{|\beta|}{\gcd(\beta,\gamma)}} \cdot \frac{1-\alpha}{\beta},
\]
and if we substitute the values of $p$ and $q$ in~(\ref{eq:uz3}), we get
\begin{equation}\label{eq:ucupz3}
[u]\smile [z_3^*] = \frac{1}{2}\sqrt{\frac{|\beta|}{\gcd(\beta,\gamma)}}\left(1-\alpha-\gamma-\alpha\gamma + \frac{(1-\beta)(1-\delta)(1-\alpha)}{\beta}\right)[w^*].
\end{equation}
But $\det\theta = 1$ and $\alpha+\delta=2$, so
$\alpha(2-\alpha)-\beta\gamma =1 \iff -\beta\gamma = (\alpha-1)^2$ and
$1-\delta=\alpha-1$. Substituting in~(\ref{eq:ucupz3}), we are left
with
\begin{equation}
[u]\smile [z_3^*] = \frac{1}{2}\sqrt{\frac{|\beta|}{\gcd(\beta,\gamma)}}(1-\alpha-\alpha\gamma-\beta\gamma)[w^*].
\end{equation}

We get the statement defining $\zeta_1 = [u]$, $\zeta_2 = [y_3^*]$,
$\xi_2 = [-kz_1^*+\ell z_2^*]$, $\xi_3 = [z_3^*]$ and $\chi =
[w^*]$. Observe that $\xi_1 = [qz_1^*+pz_2^*]$ is equal to
$\zeta_1\smile \zeta_2$.

\bigskip
Let us see now what happens if $\rank(\theta-I)=1$ and $\det\theta =
-1$. Again we assume $1+m_1\ne 0$ or $n_1\ne 0$ and write
\[
I-\theta^{-1} = \begin{bmatrix}
1+m_1 & m_2 \\
n_1 & 1+n_2
\end{bmatrix} = 
\begin{bmatrix}
qr' & pr' \\
qs' & ps'
\end{bmatrix},
\]
where the integers $p$ and $q$ satisfy $\gcd(p,q)=1$ and $p,q\ne
0$. In this case, we have $\det(I-\theta^{-1})=0 \iff \alpha+\delta =
0$. Hence $(1+m_1) + (1+n_2) = (1+\delta)+(1+\alpha) = 2 \iff
qr'+ps'=2$, so $\gcd(\beta,\gamma,2) = \gcd(r',s') \in \{1,2\}$. We
separate the analysis in two subcases.

\bigskip
\noindent{\it Third case:\/} First we see what happens when
$\gcd(\beta,\gamma,2)=\gcd(r',s')=1$. If that is the case, then
\begin{align*}
H^1(G,\zz) &= \gen{[u]}\oplus\gen{[y_3^*]} \cong \zz\oplus\zz, \\
H^2(G,\zz) &= \gen{[-kz_1^*+\ell z_2^*]} \cong \zz, \\
H^3(G,\zz) &= \gen{[w^*] \mid 2\cdot[w^*] = 0} \cong \zz_2,
\end{align*}
where $u = -s'y_1^*+r'y_2^*$ and the integers $k$ and $\ell$ satisfy
$pk+q\ell=1$. Looking at the Smith normal form of $(I-\theta^{-1})$
like we did in the previous case, we get that $[qz_1^*+pz_2^*]=0$ in
$H^2(G,\zz)$. Using equation~(\ref{eq:h1cuph1}), we obtain
\[
(u\smile u) = \dfrac{r's'(-r'-s')}{2}(qz_1^*+pz_2^*),
\]
that implies $[u]^2=0$. We also have $[y_3^*]^2=0$ and $(u\smile
y_3^*) = -s'z_1^*+r'z_2^*$.

As $pk+q\ell=1 \iff p(2k)+q(2\ell) = 2$ and $ps'+qr'=2$, there is an
integer $m$ such that
\[
\left|
\begin{array}{l}
s' = 2k-qm\\
r' = 2\ell+pm
\end{array}
\right..
\]
We know that $m$ is odd, since $\gcd(r',s')=1$. Hence
\[
[u]\smile[y_3^*] = [-2kz_1^*+2\ell z_2^*] + [mqz_1^*+mpz_2^*] = 2[-kz_1^*+\ell z_2^*].
\]

As to the products $H^1(G,\zz)\otimes H^2(G,\zz)
\stackrel{\smile}{\to} H^3(G,\zz)$, we use
equation~(\ref{eq:h1cuph2det-1}) to get that $[y_3^*]\smile
[-kz_1^*+\ell z_2^*] = 0$, and we also have
\begin{align*}
(u\smile (-kz_1^*+\ell z_2^*)) &= (r'k-s'\ell)w^* \\
&= ((2\ell + pm)k - (2k-qm)\ell)w^* \\
&= mw^*,
\end{align*}
and, since $m$ is odd, $[u]\smile[-kz_1^*+\ell z_2^*] = [w^*]$. We get
the statement of the theorem in this case by letting $\zeta_1 =[u]$,
$\zeta_2 = [y_3^*]$ and $\xi = [-kz_1^*+\ell z_2^*]$. 

\bigskip
\noindent{\it Fourth case:\/} Let us now analyze the case
$\rank(\theta-I)=1$, $\det\theta=-1$ and $\gcd(\beta,\gamma,2) =
\gcd(r',s')=2$. We have
\begin{align*}
H^1(G,\zz) &= \gen{[u]}\oplus\gen{[y_3^*]} \cong \zz\oplus\zz, \\
H^2(G,\zz) &= \gen{[qz_1^*+pz_2^*]}\oplus\gen{[-kz_1^*+\ell z_2^*]} \cong \zz_2\oplus\zz, \\
H^3(G,\zz) &= \gen{[w^*]} \cong \zz_2,
\end{align*}
where $u = -\dfrac{s'}{2}y_1^* + \dfrac{r'}{2}y_2^*$ and the integers
$k$ and $\ell$ such that $pk+q\ell=1$. Using~(\ref{eq:h1cuph1}), we get
\[
(u\smile u) = \frac{r's'(-r'-s')}{8}(qz_1^*+pz_2^*).
\]
But $r'$ and $s'$ are even, so $r's'(-r'-s') \equiv 0 \pmod{16}$,
which implies $[u]^2=0$. We also have $[y_3]^2=0$ and
\[
(u\smile y_3^*) = -\frac{s'}{2}z_1^* + \frac{r'}{2}z_2^*.
\]
There is an integer $m$ such that
\[
\left|
\begin{array}{l}
s' = 2k-qm,\\
r' = 2\ell+pm.
\end{array}
\right.
\]
In this case $m$ is even, $m=2m'$, hence
\[
(u\smile y_3^*) = (-k+qm')z_1^* + (\ell + pm')z_2^* = (-kz_1^*+\ell z_2^*) + m'(qz_1^*+pz_2^*),
\]
from where it follows that $[u]\smile[y_3^*] = m'[qz_1^*+pz_2^*] +
[-kz_1^*+\ell z_2^*]$, and the integer $m'$ may be even or odd.

We still have $[u]\smile[qz_1^*+pz_2^*]$ and $[u]\smile[-kz_1^*+\ell
  z_2^*]$ left to compute, since equation~(\ref{eq:h1cuph2det-1})
gives us at once that $[y_3^*]\smile[qz_1^*+pz_2^*] = 0$ and
$[y_3^*]\smile[-kz_1^*+\ell z_2^*]=0$. We have
\[
(u\smile (qz_1^*+pz_2^*)) = \frac{-s'p-r'q}{2}w^* = -w^*,
\]
so $[u]\smile[qz_1^*+pz_2^*]=[w^*]$. Finally,
\[
(u\smile (-kz_1^*+\ell z_2^*)) = \frac{kr'-\ell s'}{2}w^*,
\]
hence $[u]\smile [-kz_1^*+\ell z_2^*] = \dfrac{kr'-\ell s'}{2}[w^*]$,
where the integer $\dfrac{kr'-\ell s'}{2}$ may be even or odd. We get
the statement of the theorem in this case setting $\zeta_1=[u]$,
$\zeta_2=[y_3^*]$, $\xi_1 = [qz_1^*+pz_2^*]$, $\xi_2 = [-kz_1^*+\ell
  z_2^*]$. 

\bigskip
\noindent{\itshape Fifth case:\/} $\rank(\theta-I)=2$ e
$\det\theta=1$. In this case we have
\begin{align*}
H^1(G,\zz) &= \gen{[y_3^*]} \cong \zz, \\
H^2(G,\zz) &= \gen{[u] \mid c_1\cdot [u]=0}\oplus\gen{[v] \mid c_2\cdot [v]=0}\oplus\gen{[z_3^*]} \cong \zz_{c_1}\oplus\zz_{c_2}\oplus\zz, \\
H^3(G,\zz) &= \gen{[w^*]} \cong \zz,
\end{align*}
From~(\ref{eq:h1cuph1}), we get $[y_3^*]^2=0$. The generators of the
factors $\zz_{c_1}$ and $\zz_{c_2}$ of $H^2(G,\zz)$ are the classes of
maps $u$ and $v$ that are linear combinations of $z_1^*$ and $z_2^*$,
hence equation~(\ref{eq:h1cuph2det1}) shows us immediately that
$[y_3^*]\smile[u] = [y_3^*]\smile[v] = 0$ and $[y_3^*]\smile[z_3^*] =
[w^*]$. We get the statement of the theorem in this case setting
$\zeta=[y_3^*]$, $\xi_1=[u]$, $\xi_2=[v]$, $\xi_3=[z_3^*]$. 

\bigskip
\noindent{\itshape Sixth case:\/} $\rank(\theta-I)=2$ e
$\det\theta=-1$. In this case we have
\begin{align*}
H^1(G,\zz) &= \gen{[y_3^*]} \cong \zz, \\
H^2(G,\zz) &= \gen{[u]}\oplus\gen{[v]} \cong \zz_{c_1}\oplus\zz_{c_2}, \\
H^3(G,\zz) &= \gen{[w^*]} \cong \zz_2,
\end{align*}
This case is similar to the last one. We have $[y_3^*]^2=0$ and
$[y_3^*]\smile[u] = [y_3^*]\smile[v] = 0$. We get the statement
defining $\zeta=[y_3^*]$, $\xi_1=[u]$, $\xi_2=[v]$, $\chi=[w^*]$.
\end{proof}

Once we've computed the cup product in $H^*(G,\zz)$, doing the same in
$H^*(G,\zz_2)$ and $H^*(G,\zz_p)$ is easy, since essentially all we
have to do is reduce everything modulo $2$ and $p$, respectively. We
summarize the results we get in the next two theorems.

\begin{theorem}
The cohomology ring $H^*(G,\zz_2)$ is given by
\begin{enumerate}
\item If $\theta \equiv \begin{bmatrix} 1 & 0 \\ 0 & 1 \end{bmatrix}
  \pmod{2}$, then
\[
H^*(G,\zz_2) \cong \dfrac{\zz_2[\zeta_1,\zeta_2,\zeta_3]}{\begin{pmatrix}\zeta_1^2 = \left(\frac{1+m_1}{2}\right)\zeta_1\zeta_3 + \frac{m_2}{2}\zeta_2\zeta_3, \\
\zeta_2^2 = \frac{n_1}{2}\zeta_1\zeta_3 + \left(1+\frac{n_2-1}{2}\right)\zeta_2\zeta_3,\: \zeta_3^2=0 \end{pmatrix}},
\]
where $\dim(\zeta_1) = \dim(\zeta_2) = \dim(\zeta_3) = 1$.
\item If $\theta \equiv \begin{bmatrix} 0 & 1 \\ 1 & 0 \end{bmatrix}
  \pmod{2}$, then
\[
H^*(G,\zz_2) \cong \dfrac{\zz_2[\zeta_1,\zeta_2,\xi_1,\xi_2]}{\begin{pmatrix}\zeta_1^2 = \left(1+\frac{\alpha+\delta+m_2+n_1}{2}\right)\xi_1,\:\zeta_2^2=0,\:\zeta_1\zeta_2=0,\:\zeta_1\xi_1=\zeta_2\xi_2,\\
\zeta_1\xi_2 = (S(\alpha,\gamma) + S(\beta,\delta))\zeta_1\xi_1,\: \xi_1^2=0,\: \xi_2^2=0, \xi_1\xi_2=0\end{pmatrix}},
\]
where $\dim(\zeta_1) = \dim(\zeta_2) = 1$ and $\dim(\xi_1) = \dim(\xi_2) = 2$.
\item If $\theta \equiv \begin{bmatrix} 1 & 1 \\ 1 & 0 \end{bmatrix}
  \pmod{2}$, then
\[
H^*(G,\zz_2) \cong \dfrac{\zz_2[\zeta,\xi]}{(\zeta^2=0,\: \xi^2=0)},
\]
where $\dim(\zeta)=1$ and $\dim(\xi)=2$.
\item If $\theta \equiv \begin{bmatrix} 1 & 0 \\ 1 & 1 \end{bmatrix}
  \pmod{2}$, then
\[
H^*(G,\zz_2) \cong \dfrac{\zz_2[\zeta_1,\zeta_2,\xi_1,\xi_2]}{\begin{pmatrix}\zeta_1^2 = \frac{\gamma}{2}\xi_1,\: \zeta_2^2=0,\: \zeta_1\zeta_2=0,\: \zeta_1\xi_2=S(\alpha,\gamma)\zeta_1\xi_1,\\
\zeta_2\xi_1=0,\: \zeta_2\xi_2=\zeta_1\xi_1,\: \xi_1^2=0,\: \xi_2^2=0,\: \xi_1\xi_2=0\end{pmatrix}},
\]
where $\dim(\zeta_1) = \dim(\zeta_2) = 1$ and $\dim(\xi_1) = \dim(\xi_2) = 2$.
\item If $\theta \equiv \begin{bmatrix} 1 & 1 \\ 0 & 1 \end{bmatrix}
  \pmod{2}$, then
\[
H^*(G,\zz_2) \cong \dfrac{\zz_2[\zeta_1,\zeta_2,\xi_1,\xi_2]}{\begin{pmatrix}\zeta_1^2 = \frac{\beta}{2}\xi_1,\: \zeta_2^2=0,\: \zeta_1\zeta_2=0,\: \zeta_1\xi_2=S(\beta,\delta)\zeta_1\xi_1,\\
\zeta_2\xi_1=0,\: \zeta_2\xi_2=\zeta_1\xi_1,\: \xi_1^2=0,\: \xi_2^2=0,\: \xi_1\xi_2=0\end{pmatrix}},
\]
where $\dim(\zeta_1) = \dim(\zeta_2) = 1$ and $\dim(\xi_1) = \dim(\xi_2) = 2$.
\item If $\theta \equiv \begin{bmatrix} 0 & 1 \\ 1 & 1 \end{bmatrix}
  \pmod{2}$, then
\[
H^*(G,\zz_2) \cong \dfrac{\zz_2[\zeta,\xi]}{(\zeta^2=0,\: \xi^2=0)},
\]
where $\dim(\zeta)=1$ and $\dim(\xi)=2$.
\end{enumerate}
\end{theorem}

\begin{theorem}
The cohomology ring $H^*(G,\zz_p)$, for $p$ an odd prime, is given by
\begin{enumerate}
\item If $\rank_{\zz_p}(\theta-I) = 0$ and $\det\theta = 1$, then
\[
H^*(G,\zz_p) \cong \dfrac{\zz_p[\zeta_1,\zeta_2,\zeta_3]}{(\zeta_1^2=\zeta_2^2=\zeta_3^2=0)},
\]
where $\dim(\zeta_1) = \dim(\zeta_2) = \dim(\zeta_3) = 1$.
\item If $\rank_{\zz_p}(\theta-I) = 0$ and $\det\theta=-1$, then
\[
H^*(G,\zz_p) \cong \dfrac{\zz_p[\zeta_1,\zeta_2,\zeta_3]}{(\zeta_1^2=\zeta_2^2=\zeta_3^2=\zeta_1\zeta_2=0)},
\]
where $\dim(\zeta_1) = \dim(\zeta_2) = \dim(\zeta_3) = 1$.
\item If $\rank_{\zz_p}(\theta-I) = 1$ and $\det\theta=1$, then
\begin{align*}
H^1(G, \zz_p) &= \gen{[u]}\oplus\gen{[y_3^*]} \cong \zz_p\oplus\zz_p, \\
H^2(G, \zz_p) &= \dfrac{\gen{z_1^*}\oplus\gen{z_2^*}}{\im\partial_2^*}\oplus\gen{[z_3^*]} \cong \zz_p\oplus\zz_p, \\
H^3(G, \zz_p) &= \gen{[w^*]} \cong \zz_p,
\end{align*}
where $u$ is described in the following way: we have
\[
\begin{bmatrix}
m_1 & m_2 \\
n_1 & n_2
\end{bmatrix} = -\theta^{-1} =
\begin{bmatrix}
-\delta & \gamma \\
\beta & -\alpha
\end{bmatrix}.
\]
If $(1+m_1)\not\equiv 0 \pmod{p}$ or $n_1\not\equiv 0 \pmod{p}$, we
take $u = -n_1y_1^* + (1+m_1)y_2^*$ and, if $(1+m_1) \equiv n_1 \equiv
0 \pmod{p}$, we take $u = (1+n_2)y_1^* - m_2y_2^*$. Assuming, without
loss of generality, that $u = -n_1y_1^* + (1+m_1)y_2^*$, and taking
the class of $m_1z_1^*+n_1z_2^*$ as the generator of
$\dfrac{\gen{z_1^*}\oplus\gen{z_2^*}}{\im\partial_2^*} \cong \zz_p$,
we have
\[
H^*(G,\zz_p) \cong \dfrac{\zz_p[\zeta_1,\zeta_2,\xi_1,\xi_2]}{\begin{pmatrix}\zeta_1^2=0,\: \zeta_2^2=0,\: \zeta_1\zeta_2=\lambda\xi_1,\: \zeta_2\xi_1=0, \\
\zeta_1\xi_1 = (-n_1^2-(1+m_1)m_1)\zeta_2\xi_2, \\
\zeta_1\xi_2 = \left(\frac{(1-\alpha-\gamma-\alpha\gamma)}{2}(-n_1)+\frac{(1-\beta-\delta+\beta\delta)}{2}(1+m_1)\right)\zeta_2\xi_2,\\
\xi_1^2=0, \: \xi_2^2=0, \: \xi_1\xi_2=0 \end{pmatrix}},
\]
where $\zeta_1=[u]$, $\zeta_2=[y_3^*]$, $\xi_1 = [m_1z_1^*+n_1z_2^*]$,
$\xi_2 = [z_3^*]$, and $\lambda\in\zz_p$ is such that $\lambda\xi_1 =
[-n_1z_1^*+(1+m_1)z_2^*]$.
\item If $\rank_{\zz_p}(\theta-I) = 1$ and $\det\theta=-1$, then
\begin{align*}
H^1(G, \zz_p) &= \gen{[u]}\oplus\gen{[y_3^*]} \cong \zz_p\oplus\zz_p, \\
H^2(G, \zz_p) &= \dfrac{\gen{z_1^*}\oplus\gen{z_2^*}}{\im\partial_2^*} \cong \zz_p, \\
H^3(G, \zz_p) &= 0,
\end{align*}
where $u$ where is described like in the previous case and assuming,
without loss of generality, that $u = -n_1y_1^* + (1+m_1)y_2^*$, and
once again taking the class of $m_1z_1^*+n_1z_2^*$ as the generator of
$\dfrac{\gen{z_1^*}\oplus\gen{z_2^*}}{\im\partial_2^*} \cong \zz_p$,
we have 
\[
H^*(G,\zz_p) \cong \dfrac{\zz_p[\zeta_1,\zeta_2,\xi]}{\begin{pmatrix}\zeta_1^2=0,\: \zeta_2^2=0,\: \zeta_1\zeta_2=\lambda\xi\\
\zeta_1\xi=0,\: \zeta_2\xi=0,\: \xi^2=0 \end{pmatrix}},
\]
where $\zeta_1=[u]$, $\zeta_2=[y_3^*]$, $\xi = [m_1z_1^*+n_1z_2^*]$,
and $\lambda\in\zz_p$ is such that $\lambda\xi =
[-n_1z_1^*+(1+m_1)z_2^*]$.
\item If $\rank_{\zz_p}(\theta-I) = 2$ and $\det\theta=1$, then
\[
H^*(G,\zz_p) \cong \dfrac{\zz_p[\zeta,\xi]}{(\zeta^2=0, \xi^2=0)},
\]
where $\dim(\zeta)=1$ and $\dim(\xi)=2$.
\item If $\rank_{\zz_p}(\theta-I) = 2$ and $\det\theta=-1$, then
\[
H^*(G,\zz_p) \cong \dfrac{\zz_p[\zeta]}{(\zeta^2=0)},
\]
where $\dim(\zeta)=1$.
\end{enumerate}
\end{theorem}

\appendix
\section{Appendix}

The following table describes the sets $I_1,\:J_1,\:I_2,\:J_2$ that
define the element $E\in \zz G$ in Theorem~\ref{theorem:resolution},
depending on the signs of $m_1$, $n_1$, $m_2$ and $n_2$. The first
four columns indicate the sign (or the value) of the corresponding
variable. The symbol $\zz$ in one of the first four columns indicate
that the value of the corresponding variable doesn't matter for the
computation of the sets $I_1,\:J_1,\:I_2,\:J_2$.

\begin{longtable}{|c|c|c|c|p{6cm}|p{4cm}|}
\caption{The sets $I_1$, $J_1$, $I_2$ and $J_2$}\label{table:ij}\\
\hline
$m_1$ & $n_1$ & $m_2$ & $n_2$ & \hfil $I_1,\:J_1,\:I_2,\:J_2$ \hfil & Remark \\ \hline
$+$ & $+$ & $+$ & $+$ &
$I_1 = [-m_1-m_2,-m_2-1]\cap\zz$,\par
$J_1 = [-n_2,-n_1-1]\cap\zz$,\par
$I_2 = [-m_2,-m_1-1]\cap\zz$,\par
$J_2 = [-n_1,-1]\cap\zz$ & if $m_1<m_2$, $n_1<n_2$ \\ \hline
$+$ & $+$ & $+$ & $+$ &
$I_1 = [-m_1,-m_2-1]\cap\zz$,\par
$J_1 = [-n_2,-1]\cap\zz$,\par
$I_2 = [-m_1-m_2,-m_1-1]\cap\zz$,\par
$J_2 = [-n_1,-n_2-1]\cap\zz$ & if $m_1>m_2$, $n_1>n_2$\\ \hline
$-$ & $-$ & $-$ & $-$ &
$I_1 = [-m_2,-m_1-m_2-1]\cap\zz$,\par
$J_1 = [-n_1,-n_2-1]\cap\zz$,\par
$I_2 = [-m_1,-m_2-1]\cap\zz$,\par
$J_2 = [0,-n_1-1]\cap\zz$ & if $m_1>m_2$, $n_1>n_2$\\ \hline
$-$ & $-$ & $-$ & $-$ &
$I_1 = [-m_2,-m_1-1]\cap\zz$,\par
$J_1 = [0,-n_2-1]\cap\zz$,\par
$I_2 = [-m_1,-m_1-m_2-1]\cap\zz$,\par
$J_2 = [-n_2,-n_1-1]\cap\zz$ & if $m_1<m_2$, $n_1<n_2$\\ \hline
$-$ & $+$ & $+$ & $-$ &
$I_1 = [-m_2,-m_1-m_2-1]\cap\zz$,\par
$J_1 = [0,-n_2-1]\cap\zz$,\par
$I_2 = [-m_1-m_2,-m_1-1]\cap\zz$,\par
$J_2 = [-n_1,-1]\cap\zz$ & \\ \hline
$+$ & $-$ & $-$ & $+$ & 
$I_1 = [-m_1-m_2,-m_2-1]\cap\zz$,\par
$J_1 = [-n_2,-1]\cap\zz$,\par
$I_2 = [-m_1,-m_1-m_2-1]\cap\zz$,\par
$J_2 = [0,-n_1-1]\cap\zz$ & \\ \hline
$+$ & $-$ & $+$ & $-$ &
$I_1 = [-m_2,-m_1-1]\cap\zz$,\par
$J_1 = [0,-n_1-1]\cap\zz$,\par
$I_2 = [-m_1-m_2,-m_2-1]\cap\zz$,\par
$J_2 = [-n_1,-n_2-1]\cap\zz$ & if $m_1<m_2$, $n_1>n_2$\\ \hline
$+$ & $-$ & $+$ & $-$ & 
$I_1 = [-m_1-m_2,-m_1-1]\cap\zz$,\par
$J_1 = [-n_2,-n_1-1]\cap\zz$,\par
$I_2 = [-m_1,-m_2-1]\cap\zz$,\par
$J_2 = [0,-n_2-1]\cap\zz$ & if $m_1>m_2$, $n_1<n_2$\\ \hline
$-$ & $+$ & $-$ & $+$ & 
$I_1 = [-m_1,-m_1-m_2-1]\cap\zz$,\par
$J_1 = [-n_1,-n_2-1]\cap\zz$,\par
$I_2 = [-m_2,-m_1-1]\cap\zz$,\par
$J_2 = [-n_2,-1]\cap\zz$ & if $m_1<m_2$, $n_1>n_2$\\ \hline
$-$ & $+$ & $-$ & $+$ & 
$I_1 = [-m_1,-m_2-1]\cap\zz$,\par
$J_1 = [-n_1,-1]\cap\zz$,\par
$I_2 = [-m_2,-m_1-m_2-1]\cap\zz$,\par
$J_2 = [-n_2,-n_1-1]\cap\zz$ & if $m_1>m_2$, $n_1<n_2$\\ \hline
$+$ & $+$ & $-$ & $-$ & 
$I_1 = [-m_1,-m_1-m_2-1]\cap\zz$,\par
$J_1 = [-n_1,-1]\cap\zz$,\par
$I_2 = [-m_1-m_2,-m_2-1]\cap\zz$,\par
$J_2 = [0,-n_2-1]\cap\zz$ & \\ \hline
$-$ & $-$ & $+$ & $+$ & 
$I_1 = [-m_1-m_2,-m_1-1]\cap\zz$,\par
$J_1 = [0,-n_1-1]\cap\zz$,\par
$I_2 = [-m_2,-m_1-m_2-1]\cap\zz$,\par
$J_2 = [-n_2,-1]\cap\zz$ & \\ \hline
$0$ & $1$ & $1$ & $\zz$ &
$I_1 = \emptyset$,\par
$J_1 = \emptyset$,\par
$I_2 = \{-1\}$,\par
$J_2 = \{-1\}$ & \\ \hline
$0$ & $-1$ & $-1$ & $\zz$ &
$I_1 = \emptyset$,\par
$J_1 = \emptyset$,\par
$I_2 = \{0\}$,\par
$J_2 = \{0\}$ & \\ \hline
$0$ & $1$ & $-1$ & $\zz$ &
$I_1 = \{0\}$,\par
$J_1 = \{-1\}$,\par
$I_2 = \emptyset$,\par
$J_2 = \emptyset$ & \\ \hline
$0$ & $-1$ & $1$ & $\zz$ &
$I_1 = \{-1\}$,\par
$J_1 = \{0\}$,\par
$I_2 = \emptyset$,\par
$J_2 = \emptyset$ & \\ \hline
$1$ & $0$ & $\zz$ & $1$ &
$I_1 = \{-m_2-1\}$,\par
$J_1 = \{-1\}$,\par
$I_2 = \emptyset$,\par
$J_2 = \emptyset$ & \\ \hline
$-1$ & $0$ & $\zz$ & $-1$ &
$I_1 = \{-m_2\}$,\par
$J_1 = \{0\}$,\par
$I_2 = \emptyset$,\par
$J_2 = \emptyset$ & \\ \hline
$1$ & $0$ & $\zz$ & $-1$ &
$I_1 = \emptyset$,\par
$J_1 = \emptyset$,\par
$I_2 = \{-m_2-1\}$,\par
$J_2 = \{0\}$ & \\ \hline
$-1$ & $0$ & $\zz$ & $1$ &
$I_1 = \emptyset$,\par
$J_1 = \emptyset$,\par
$I_2 = \{-m_2\}$,\par
$J_2 = \{-1\}$ & \\ \hline
$1$ & $\zz$ & $0$ & $1$ &
$I_1 = \{-1\}$,\par
$J_1 = \{-1\}$,\par
$I_2 = \emptyset$,\par
$J_2 = \emptyset$ & \\ \hline
$-1$ & $\zz$ & $0$ & $-1$ &
$I_1 = \{0\}$,\par
$J_1 = \{0\}$,\par
$I_2 = \emptyset$,\par
$J_2 = \emptyset$ & \\ \hline
$1$ & $\zz$ & $0$ & $-1$ &
$I_1 = \emptyset$,\par
$J_1 = \emptyset$,\par
$I_2 = \{-1\}$,\par
$J_2 = \{0\}$ & \\ \hline
$-1$ & $\zz$ & $0$ & $1$ &
$I_1 = \emptyset$,\par
$J_1 = \emptyset$,\par
$I_2 = \{0\}$,\par
$J_2 = \{-1\}$ & \\ \hline
$\zz$ & $1$ & $1$ & $0$ &
$I_1 = \emptyset$,\par
$J_1 = \emptyset$,\par
$I_2 = \{-m_1-1\}$,\par
$J_2 = \{-1\}$ & \\ \hline
$\zz$ & $-1$ & $-1$ & $0$ &
$I_1 = \emptyset$,\par
$J_1 = \emptyset$,\par
$I_2 = \{-m_1\}$,\par
$J_2 = \{0\}$ & \\ \hline
$\zz$ & $1$ & $-1$ & $0$ &
$I_1 = \{-m_1\}$,\par
$J_1 = \{-1\}$,\par
$I_2 = \emptyset$,\par
$J_2 = \emptyset$ & \\ \hline
$\zz$ & $-1$ & $1$ & $0$ &
$I_1 = \{-m_1-1\}$,\par
$J_1 = \{0\}$,\par
$I_2 = \emptyset$,\par
$J_2 = \emptyset$ & \\ \hline
\end{longtable}

\bibliographystyle{plain}
\bibliography{references}

\end{document}